\documentclass[a4paper,reqno]{amsart}
\usepackage{latexsym,amsmath,amssymb,amsfonts}
\usepackage{a4wide}
\usepackage{color}
\usepackage{graphicx}
\usepackage[bookmarks]{hyperref} 

\theoremstyle{plain}
\begingroup
\newtheorem{theorem}{Theorem}[section]
\newtheorem{lemma}[theorem]{Lemma}
\newtheorem{proposition}[theorem]{Proposition}
\newtheorem{corollary}[theorem]{Corollary}
\endgroup
\theoremstyle{definition}
\begingroup
\newtheorem{definition}[theorem]{Definition}
\newtheorem{remark}[theorem]{Remark}
\newtheorem{example}[theorem]{Example}
\endgroup
\theoremstyle{remark}

\mathchardef\emptyset="001F

\numberwithin{equation}{section}

\newcommand{\e}{\varepsilon}
\newcommand{\vphi}{\varphi}
\newcommand{\vt}{\vartheta}
\newcommand{\N}{\mathbb N}

\newcommand{\R}{\mathbb R}
\newcommand{\C}{\mathbb C}

\renewcommand{\div}{{\rm div}}
\newcommand{\dist}{{\rm dist\,}}
\newcommand{\supp}{{\rm supp\,}}

\newcommand{\dev}{{\rm Dev}}
\newcommand{\gr}{{\rm gr}}

\newcommand{\hu}{{\mathcal H}^{1}}
\newcommand{\huu}{{\mathcal H}^{0}}
\newcommand{\wto}{\rightharpoonup}
\newcommand{\setmeno}{\!\setminus\!}

\newdimen\mex
\makeatletter
\def\niv{\mathrel{\hbox{\hglue -0.4\mex
\vrule \@height 1\mex \@width .1\mex
\vrule \@height .1\mex \@width 1\mex
\hglue -0.2\mex}}}
\makeatother

\newcommand{\F}{\mathcal{F}}
\newcommand{\sbv}{SBV}

\newcommand{\ombar}{\overline{\Omega}}
\newcommand{\spazio}{\mathcal{A}(\Omega)}
\newcommand{\spazioreg}{\mathcal{A}_{reg}(\Omega)}

\newcommand{\dir}{H^{1}_U(\Omega\setminus \Gamma)}
\newcommand{\flusso}{{\Phi_t}}
\newcommand{\bordo}{K\cap\partial\Omega}

\newcommand{\bordoft}{K_{\Phi_t}\cap\partial\Omega}

\newcommand{\nbordo}{\Gamma\cap\partial\Omega}

\newcommand{\nbordot}{\Gamma_t\cap\partial\Omega}
\newcommand{\kf}{K_\Phi}
\newcommand{\kft}{K_{\Phi_t}}

\newcommand{\ncampot}{X}

\newcommand{\nfield}{X_\psi}

\newcommand{\nuft}{\nu_{\Phi_t}}
\newcommand{\etaf}{\eta_\Phi}
\newcommand{\etaft}{\eta_{\Phi_t}}
\newcommand{\norm}{\nu_{\partial\Omega}}
\newcommand{\vf}{v_\vphi}

\title[Stable critical points of the Mumford-Shah functional]{Stable regular critical points of the Mumford-Shah functional are local minimizers}
\author{M.\ Bonacini}
\author{M.\ Morini}
\address[M.~Bonacini]{SISSA, Via Bonomea 265, 34136 Trieste, Italy}
\email{marco.bonacini@sissa.it}
\address[M.~Morini]{Dipartimento di Matematica, Universit\`{a} degli Studi di Parma, Parma, Italy}
\email{massimiliano.morini@unipr.it}

\thanks{Preprint SISSA 33/2013/MATE}

\begin{document}

\begin{abstract}
In this paper it is shown that any regular critical point of the Mumford-Shah functional, with positive definite second variation,
is an isolated local minimizer with respect to competitors which are sufficiently close in the $L^1$-topology.
\end{abstract}

\maketitle


\begin{section}{Introduction} \label{sect:intro}

The \emph{Mumford-Shah functional} is the most typical example of a class of variational problems
called by E. De Giorgi \emph{free discontinuity problems}, characterized by the competition between volume and surface terms.
The minimization of such an energy was proposed in the seminal papers \cite{MS1,MS2} in the context of image segmentation,
and plays an important role also in variational models for fracture mechanics.
Its homogeneous version in a bounded open set $\Omega\subset\R^2$ is defined over pairs $(\Gamma,u)$, with $\Gamma$ closed subset of $\overline{\Omega}$ and $u\in H^1(\Omega\setminus\Gamma)$, as
\begin{equation} \label{intro}
F(\Gamma,u) := \int_{\Omega\setminus\Gamma} |\nabla u|^2\,dx + \mathcal{H}^{1}(\Gamma\cap\Omega).
\end{equation}
Since its introduction, several results concerning the existence and regularity of minimizers,
as well as the structure of the optimal set, have been obtained
(see, \emph{e.g.}, \cite{AFP} for a detailed account on this topic).

In this paper we continue the study of second order optimality conditions for the functional in \eqref{intro}
initiated by F. Cagnetti, M.G. Mora and the second author in \cite{CMM},
where a suitable notion of second variation was introduced
by considering one-parameter families of perturbations of the regular part of the discontinuity set.
In \cite{CMM} it was also shown that a critical point $(\Gamma,u)$ with positive definite second variation
minimizes the functional with respect to pairs of the form $(\Phi(\Gamma),v)$,
where $\Phi$ is any diffeomorphism sufficiently close to the identity in the $C^2$-norm,
with $\Phi-Id$ compactly supported in $\Omega$,
and $v\in H^1(\Omega\setminus\Phi(\Gamma))$ satisfies $v=u$ on $\partial\Omega$.

In the main theorem of this paper we strongly improve the aforementioned result,
by showing that in fact the positive definiteness of the second variation
implies strict local minimality with respect to the weakest topology which is natural for this problem,
namely the $L^1$-topology.
To be more precise, we prove that if $(\Gamma,u)$ is a critical point with positive second variation,
then there exists $\delta>0$ such that
\begin{equation*}
F(\Gamma,u) < F(K,v)
\end{equation*}
for all admissible pairs $(K,v)$, provided that $v$ attains the same boundary conditions as $u$
and $0<\|u-v\|_{L^1(\Omega)}<\delta$.
We mention that for technical reasons the boundary conditions imposed here are slightly different from those considered in \cite{CMM}, as we prescribe the Dirichlet condition only on a portion $\partial_D\Omega\subset\partial\Omega$ away from the intersection of the discontinuity set $\Gamma$ with $\partial\Omega$.

The general strategy of the proof is close in spirit to the one devised in \cite{FM} for a different free-discontinuity problem.
It consists in two fundamental steps:
first, one shows that strict stability is sufficient to guarantee local minimality with respect to perturbations of the discontinuity set which are close to the identity in the $W^{2,\infty}$-norm (see Theorem~\ref{teo:minW}).
This amounts to adapting to our slightly different context the techniques developed in \cite{CMM},
with the main new technical difficulties stemming from allowing also boundary variations of the discontinuity set.

The second step of the outline consists in showing that the above local $W^{2,\infty}$-minimality
in fact implies the claimed local $L^1$-minimality.
This is done through a penalization/regularization approach,
with an appeal to the regularity theory of quasi-minimizers of the area functional and of the Mumford-Shah functional (see \cite{AFP}).
More precisely, we start by showing that the local $W^{2,\infty}$-minimality
implies minimality with respect to small $C^{1,\alpha}$-perturbations of the discontinuity set.
This is perhaps the most technical part of the proof.
The main idea is to restrict $F$ to the class of pairs $(\Gamma,v)$
such that $\|v-u\|_{W^{1,\infty}(\Omega\setminus\Gamma)}\leq1$,
so that the Dirichlet energy behaves like a volume term,
and $F$ can be regarded as a volume perturbation of the area functional.
This allows to use the regularity theory for quasi-minimizers of the area functional to deduce the local $C^{1,\alpha}$-minimality through a suitable contradiction argument.

A contradiction argument is also finally used to establish the sought $L^1$-minimality.
To give a flavor of this type of reasoning, we sketch here the main steps of this last part of the proof.
One assumes by contradiction the existence of admissible pairs $(\Gamma_n,u_n)$
with $u_n$ converging to $u$ in $L^1(\Omega)$, such that the minimality inequality fails along the sequence:
\begin{equation} \label{eqintro}
F(\Gamma_n,u_n)\leq F(\Gamma,u)
\end{equation}
for every $n$.
By an easy truncation argument, we may also assume that $\|u_n\|_\infty\leq\|u\|_\infty$,
so that $u_n\to u$ in $L^p(\Omega)$ for every $p\geq1$.
Then we replace each $(\Gamma_n,u_n)$ by a new pair $(K_n,v_n)$ chosen as solution to a suitable penalization problem, namely
$$
\min \Bigl\{ F(K,w) + \beta \Bigl( \sqrt{(\|w-u\|_{L^2(\Omega)}^2-\e_n)^2+\e_n^2} -\e_n \Bigr)
\, : \, (K,w) \text{ admissible, } w=u \text{ on } \partial_D\Omega \Bigr\} \,,
$$
with $\e_n:=\|u_n-u\|_{L^2(\Omega)}^2\to0$, and $\beta>0$ large enough.
Note that, by \eqref{eqintro} and by minimality, we have
\begin{equation} \label{intro3}
F(K_n,v_n)\leq F(\Gamma_n,u_n)\leq F(\Gamma,u).
\end{equation}
The advantage is now that the pairs $(K_n,v_n)$ satisfy a uniform \emph{quasi-minimality} property (see Theorem~\ref{teo:parreg}).
It is easy to show that, up to subsequences, the sequence $(K_n,v_n)$ converges to a minimizer of the limiting problem
\begin{equation} \label{intro2}
\min \Bigl\{ F(K,w) + \beta \|w-u\|_{L^2(\Omega)}^2
\, : \, (K,w) \text{ admissible, } w=u \text{ on } \partial_D\Omega \Bigr\} \,.
\end{equation}
Now a calibration argument developed in \cite{Mor} implies that we may choose $\beta$ so large that
$(\Gamma,u)$ is the unique global minimizer of \eqref{intro2}.
With this choice of $\beta$ we have in particular that $v_n\to u$ in $L^1$,
and in turn, by exploiting the regularity properties of quasi-minimizers of the Mumford-Shah functional, we infer that the corresponding discontinuity sets $K_n$ are locally $C^{1,\alpha}$-graphs and converge in the $C^{1,\alpha}$-sense to $\Gamma$. Recalling \eqref{intro3}, we have reached a contradiction to the $C^{1,\alpha}$-minimality.

We remark that a similar two-steps strategy has been used also in \cite{AFM}
for a nonlocal isoperimetric problem related to the modeling of diblock copolymers,
and in \cite{CLeo}, where the appeal to the regularity of quasi-minimizers
appears for the first time in the context of isoperimetric inequalities.

We regard our result as a first step of a more general study of second order minimality conditions for free-discontinuity problems.
Besides considering more general functionals,
it would be very interesting to extend our local minimality criterion to the case of discontinuity sets with singular points,
like the so-called ``triple junction'', where three lines meet forming equal angles of $2\pi/3$, and the ``crack-tip'', where a line terminates at some point.
This will be the subject of future investigation.

\smallskip
The paper is organized as follows.
In Section~\ref{sect:prelim} we fix the notation and we review some preliminary results concerning the regularity theory for quasi-minimizers of the Mumford-Shah functional.
In Section~\ref{sect:setting} we collect the necessary definitions and state the main result.
Section~\ref{sect:var2} is devoted to the computation of the second variation,
when also boundary variations of the discontinuity set are allowed.
The proof of the main theorem starts in Section~\ref{sect:minW}
(where the local $W^{2,\infty}$-minimality is addressed)
and lasts for Sections~\ref{sect:minW1inf} and \ref{sect:minSBV}
(where the $C^{1,\alpha}$ and the desired local $L^1$-minimality, respectively, are established).
In Section~\ref{sect:example} we describe some examples and applications of our minimality criterion.
In the final Appendix (Section~\ref{sect:appendix}) we prove some auxiliary technical lemmas needed in the paper.

\end{section}


\begin{section}{Notation and preliminaries} \label{sect:prelim}

In this section we fix the notation and we recall some preliminary results.

\subsection{Geometric preliminaries} \label{sect:prelgeom}
Let $\Gamma$ be a smooth embedded curve in $\R^2$,
let $\nu:\mathcal{U}\to\mathbb{S}^1$ be a smooth vector field
defined in a tubular neighborhood $\mathcal{U}$ of $\Gamma$ and normal to $\Gamma$ on $\Gamma$,
and let $\tau:=\nu^\bot$ be the unit tangent vector to $\Gamma$ (where $^\bot$ stands for the clockwise rotation by $\frac\pi2$).
If $g:\mathcal{U}\to \R^d$ is a smooth function,
we denote by $D_\Gamma g(x)$ ($\nabla_\Gamma g(x)$ if $d=1$) the {\em tangential differential} of $g$ at $x\in \Gamma$,
that is, the linear operator from $\R^2$ into $\R^d$ given by $D_\Gamma g(x):=dg(x)\circ \pi_x$,
where $dg(x)$ is the usual differential of $g$ at $x$
and $\pi_x$ is the orthogonal projection on the tangent space to $\Gamma$ at $x$.
If $g$ is a vector field from $\Gamma$ to $\R^2$
we define also its {\em tangential divergence} as
$\div_{\Gamma} g := \tau{\,\cdot\,}\partial_{\tau}g$.

The following {\em divergence formula} is a particular case of \cite[7.6]{Sim}:
for every smooth vector field $g:\mathcal{U}\to \R^2$ holds
\begin{equation}\label{fdiv}
\int_{\Gamma} \div_{\Gamma} g \,d\hu =
\int_{\Gamma} H  (g{\,\cdot\,}\nu)\, d\hu + \int_{\partial \Gamma} g\cdot\eta \,d\huu.
\end{equation}
Here $\partial\Gamma$ stands for the endpoints of $\Gamma$,
$\eta$ is a unit vector tangent to $\Gamma$ and pointing out of $\Gamma$ at each point of $\partial \Gamma$
(it coincides with $\tau$, up to a sign)
and the function $H$ is defined in $\mathcal{U}$ by $H:=\div\nu$.
Notice that $H$ coincides, when restricted to $\Gamma$, with the {\em curvature} of $\Gamma$
and, since $\partial_\nu\nu=0$, we have $H = \div_{\Gamma}\nu = D\nu[\tau,\tau]$.

Let $\Phi:\overline{\mathcal{U}}\to\overline{\mathcal{U}}$ be a smooth orientation-preserving diffeomorphism
and let $\Gamma_{\Phi}:=\Phi(\Gamma)$.
A possible choice for the unit normal to $\Gamma_\Phi$ is given by the vector field
\begin{equation}\label{nuf}
\nu_\Phi
= \frac{(D\Phi)^{-T}  [\nu]}{|(D\Phi )^{-T} [\nu]|}\circ\Phi^{-1},
\end{equation}
while the vector $\eta$ appearing in \eqref{fdiv} becomes
\begin{equation} \label{etaf}
\etaf =
\frac { D\Phi[\eta] } { |D\Phi[\eta]| }
\circ\Phi^{-1}
\end{equation}
on $\partial \Gamma_\Phi$.
We denote by $H_{\Phi}$ the curvature of $\Gamma_\Phi$.
We shall use the following identity,
which is a particular case of the so-called generalized area formula
(see, \textit{e.g.}, \cite[Theorem~2.91]{AFP}):
for every $\psi\in L^1(\Gamma_\Phi)$
\begin{equation}\label{farea}
\int_{\Gamma_\Phi} \psi \,d\hu =\int_\Gamma (\psi\circ\Phi) J_\Phi \, d\hu,
\end{equation}
where $J_\Phi:=|(D\Phi)^{-T}[\nu]|\det D\Phi$ is the $1$-dimensional Jacobian of $\Phi$.

\subsection{Partial regularity for quasi-minimizers of the Mumford-Shah functional} \label{sect:prelparreg}
Given an open set $\Omega\subset\R^2$, we recall that the space $\sbv(\Omega)$ of special functions of bounded variation
is defined as the set of all functions $u:\Omega\to\R$
whose distributional derivative $Du$ is a bounded Radon measure of the form
$$
Du=\nabla u \, \mathcal{L}^2 + D^ju = \nabla u \, \mathcal{L}^2 + (u^+-u^-)\nu_u\hu\niv S_u,
$$
where $\nabla u \in L^1(\Omega;\R^2)$ is the approximate gradient of $u$,
$S_u$ is the jump set of $u$ (which is countably $(\hu,1)$-rectifiable),
$u^+$ and $u^-$ are the traces of $u$ on $S_u$
and $\nu_u$ is the approximate normal on $S_u$.
We refer to \cite{AFP} for a complete treatment of the space $\sbv$
and a precise definition of all the notions introduced above.
In the sequel we will consider the following notion of convergence in the space $\sbv$,
motivated by the compactness theorem \cite[Theorem~4.8]{AFP}.

\begin{definition} \label{def:convsbv}
We say that $u_n\to u$ in $\sbv(\Omega)$ if
$u_n\to u$ strongly in $L^1(\Omega)$,
$\nabla u_n \wto \nabla u$ weakly in $L^2(\Omega;\R^N)$,
and $D^ju_n\wto D^ju$ weakly* in the sense of measures in $\Omega$.
\end{definition}

Given $u\in\sbv(\Omega)$, we introduce the quantities
$$
D_u(x,r):=\int_{B_r(x)\cap\Omega}|\nabla u|^2\,dy, \qquad
A_u(x,r):=\min_{T\in\mathcal{A}} \int_{\overline{S}_{u}\cap B_r(x)}\dist^2(y,T)\,d\hu(y),
$$
where $\mathcal{A}$ denotes the set of affine lines in $\R^2$, and
$$
E_u(x,r):= D_u(x,r) + r^{-2}A_u(x,r).
$$
The result that we are going to recall expresses the fact that the rate of decay of $E_u$ in small balls determines the $C^{1,\alpha}$-regularity of the jump set of $u$, provided that $u$ satisfies a quasi-minimality property.
In order to state precisely the theorem, we introduce some more notation.
We set $C_{\nu,r}:=\{x\in\R^2: |\pi_{\nu}(x)|<r, \, |x\cdot\nu|<r\}$ for $\nu\in\mathbb{S}^1$ and $r>0$,
where $\pi_{\nu}(x)=x-(x\cdot\nu)\nu$.
If $g:(-r,r)\to\R$, we define the \emph{graph} of $g$ (with respect to the direction $\nu$) to be the set
$$
\gr_\nu(g) := \{ x=x'+g(x')\nu \in\R^2 : \, x'=\pi_\nu(x), \, |x'|<r \} \,.
$$

\begin{theorem} \label{teo:parreg}
Let $u\in\sbv(\Omega)$ be a quasi-minimizer of the Mumford-Shah functional,
that is, assume that there exists $\omega>0$ such that for every ball $B_\rho(x)$
\begin{equation} \label{quaminMS}
\int_{\Omega\cap B_\rho(x)}|\nabla u|^2\,dx + \hu(S_u\cap B_\rho(x))
\leq \int_{\Omega\cap B_\rho(x)}|\nabla v|^2\,dx + \hu(S_v\cap B_\rho(x)) + \omega\rho^2
\end{equation}
for every $v\in\sbv(\Omega)$ with $\{v\neq u\}\subset\subset B_\rho(x)$.
There exist constants $R_0>0$, $\e_0>0$ (depending only on $\omega$) such that if
$$
E_u(x,r)<\e_0r
$$
for some $x\in\overline{S}_u\cap\Omega$ and $r<R:=R_0\wedge\dist(x,\partial\Omega)$,
then there exist a smaller radius $r'\in(0,r)$ (depending only on $\omega$, $R$ and $r$)
and a function $f\in C^{1,\frac14}(-r',r')$ with $f(0)=f'(0)=0$ such that
$$
(\overline{S}_u-x)\cap C_{\nu,r'} = \gr_{\nu}(f),
$$
where $\nu$ denotes the normal to $S_u$ at $x$. Moreover, $\|f\|_{C^{1,\frac14}}\leq C$
for some constant $C$ depending only on $\omega$.
\end{theorem}

The previous result (which holds also in dimension $N>2$) is a consequence of \cite[Theorem~8.2 and Theorem~8.3]{AFP}:
the only missing point is the uniform bound in $C^{1,\frac14}$,
which is not explicitly stated but can be deduced by checking that
the constants appearing in the proof depend only on $\omega$.
Notice that the theorem provides the regularity of $S_u$ in balls well contained in $\Omega$;
concerning the regularity of the discontinuity set at the intersection with the boundary of $\Omega$,
under Neumann conditions, we have the following result, which is essentially contained in the book \cite{Dav}
(see, in particular, \cite[Remark~79.42]{Dav}; see also \cite{MadSol}).

\begin{theorem}\label{teo:parreg2}
Let $\Omega\subset\R^2$ be a bounded, open set with boundary of class $C^1$,
and let $u\in\sbv(\Omega)$ satisfy the same assumption of Theorem~\ref{teo:parreg}.
Then there exist $b\in(0,1)$ and $\tau>0$ (depending only on $\omega$ and on $\Omega$)
such that, setting
$$
\Omega(\tau) := \{x\in\Omega : \dist(x,\partial\Omega)<\tau\},
$$
the intersection $\overline{S}_u\cap\Omega(\tau)$ is a finite disjoint union of curves of class $C^{1,b}$
intersecting $\partial\Omega$ orthogonally,
with $C^{1,b}$-norm uniformly bounded by a constant depending only on $\omega$ and $\Omega$.
\end{theorem}

We conclude this preliminary section by recalling a well known property
of quasi-minimizers of the Mumford-Shah functional,
namely a lower bound on the $\hu$-dimensional density of the jump set in balls centered at any point of its closure.
The estimate was proved in \cite{DCL} in balls entirely contained in the domain $\Omega$ (see also \cite[Theorem~7.21]{AFP});
we refer also, when a Dirichlet condition is assumed at the boundary of the domain,
to \cite{CL} for balls centered at $\partial\Omega$,
and to \cite{BG} for balls possibly intersecting $\partial\Omega$ but not necessarily centered at $\partial\Omega$,
and finally to \cite[Section~77]{Dav} in the case of balls intersecting $\partial\Omega$ when a Neumann condition is imposed.

In fact, for our purposes we will need to consider the mixed situation,
where we impose a Dirichlet condition on a part $\partial_D\Omega$ of the boundary
and a Neumann condition on the remaining part $\partial_N\Omega$.
The result is still valid in this case,
for balls centered at the intersection between the Dirichlet and the Neumann part of the boundary,
under the additional assumption that $\partial_D\Omega$ and $\partial_N\Omega$ meet orthogonally.
We are not aware of any result of this kind in the existing literature,
but the proof can be obtained by following closely the strategy of the original proof in \cite{DCL},
combined also with some new ideas contained in \cite{BG}.
We will sketch the proof in Section~\ref{appendix:lowerbound},
referring the reader to \cite{Bon} for the details.

The precise statement is the following.

\begin{theorem} \label{teo:dlb}
Let $\Omega\subset\R^2$ be a bounded, open set,
let $\partial_D\Omega\subset\partial\Omega$ be relatively open and of class $C^1$,
$\partial_N\Omega:=\partial\Omega\setminus\overline{\partial_D\Omega}$ of class $C^1$,
and assume that $\partial_D\Omega$ meets $\partial_N\Omega$ orthogonally.
Let $\Omega'\subset\R^2$ be a bounded, open set of class $C^1$ such that $\Omega\subset\Omega'$ and $\partial\Omega\cap\Omega'=\partial_D\Omega$.
Let $u\in\sbv(\Omega')$ be such that $\overline{S}_u\cap\overline{\partial_D\Omega}=\emptyset$
and $u\in W^{1,\infty}(\Omega'\setminus\overline{S}_u)$.

Let $w\in\sbv(\Omega')$, with $w=u$ in $\Omega'\setminus\Omega$, satisfy for every $x\in\ombar$ and for every $\rho>0$
$$
\int_{\Omega'\cap B_\rho(x)}|\nabla w|^2\,dx + \hu(S_w\cap B_\rho(x))
\leq \int_{\Omega'\cap B_\rho(x)}|\nabla v|^2\,dx + \hu(S_v\cap B_\rho(x)) + \omega\rho^2
$$
for every $v\in\sbv(\Omega')$ such that $v=u$ in $\Omega'\setminus\Omega$ and $\{v\neq w\}\subset\subset B_\rho(x)$.
Then there exist $\rho_0>0$ and $\theta_0>0$ (depending only on $\omega$, $u$ and $\Omega$) such that
$$
\hu(S_w\cap B_\rho(x)) \geq \theta_0\rho
$$
for every $\rho\leq\rho_0$ and $x\in\overline{S}_w$.
\end{theorem}

\end{section}


\begin{section}{Setting and main result} \label{sect:setting}

Let $\Omega\subset\R^2$ be an open, bounded, connected set with boundary of class $C^3$.
We introduce the following space of \emph{admissible pairs}
$$
\spazio := \bigl\{ (K,v) : K\subset\R^2 \text{ closed},\, v\in H^{1}(\Omega\setminus K) \bigr\}
$$
and we set
$$
F(K,v) := \int_{\Omega\setminus K} |\nabla v|^2\,dx + \hu(K\cap\Omega)
\qquad\text{for } (K,v)\in\spazio.
$$
It will be useful to consider also a localized version of the functional:
for $A\subset\Omega$ open we set
$$
F((K,v);A) := \int_{A\setminus K} |\nabla v|^2\,dx + \hu(K\cap A).
$$

Given an admissible pair $(K,v)\in\spazio$ and assuming that $K$ is a regular curve connecting two points of $\partial\Omega$,
we denote by $\nu$ a smooth vector field coinciding with the unit normal to $K$ when restricted to the points of $K$,
and by $H$ the curvature of $K$ (see Section~\ref{sect:prelgeom}).
For any function $z\in H^{1}(\Omega\setminus K)$
we denote the traces of $z$ on the two sides of $K$ by $z^+$ and $z^-$:
precisely, for $\hu$-a.e. $x\in K$ we set
$$
z^\pm(x) := \lim_{r\to0^+} \frac{1}{|B_r(x)\cap V_x^\pm|} \int_{B_r(x)\cap V_x^\pm}z(y)\,dy,
$$
where $V_x^\pm:=\{y\in\R^2:\pm (y-x)\cdot\nu(x)\geq 0\}$.
With an abuse of notation, we denote by $z^+$ and $z^-$ also the restrictions of $z$ to $\Omega^+$ and $\Omega^-$ respectively,
where $\Omega^+$ and $\Omega^-$ are the two connected components of $\Omega\setminus K$,
with the normal vector field $\nu$ pointing into $\Omega^+$.
Finally we denote by $\norm$ the exterior unit normal vector to $\partial\Omega$
and by $H_{\partial\Omega}$ the curvature of $\partial\Omega$ with respect to $\norm$.

\begin{definition} \label{def:spazioreg}
We say that $(K,v)\in\spazio$ is a \emph{regular pair} if $K$ is a curve of class $C^\infty$ connecting two points of $\partial\Omega$, and there exists $\partial_D\Omega\subset\subset\partial\Omega\setminus K$ relatively open in $\partial\Omega$
such that $v$ is a solution to
\begin{equation} \label{minel}
\int_{\Omega\setminus K} \nabla v \cdot \nabla z \,dx = 0 \qquad\text{for every } z\in H^1(\Omega\setminus K)\text{ with } z=0 \text{ on } \partial_D\Omega,
\end{equation}
that is, $v$ is a weak solution to
$$
\left\{
\begin{array}{ll}
\Delta v=0 & \hbox{in }\Omega\setminus K, \\
\partial_\nu v^\pm =0 & \hbox{on }K\cap\Omega, \\
\partial_{\norm} v =0 & \hbox{on }\partial_N\Omega:=\partial\Omega\setminus\partial_D\Omega.
\end{array}
\right.
$$
We denote by $\spazioreg$ the space of all such pairs.
\end{definition}

\begin{definition} \label{def:subdomain}
Given a regular pair $(K,v)\in\spazioreg$,
we say that an open subset $U\subset\R^2$ with Lipschitz boundary
is an \emph{admissible subdomain}
if $K\subset U$ and $\overline{U}\cap\mathcal{S}=\emptyset$,
where $\mathcal{S}$ denotes the relative boundary of $\partial_D\Omega$ in $\partial\Omega$.
In this case we define the space $H^{1}_U(\Omega\setminus K)$
consisting of all functions $v\in H^1(\Omega\setminus K)$
such that $v=0$ in $(\Omega\setminus U)\cup\partial_D\Omega$
(the condition on $\partial_D\Omega$ has to be intended in the sense of traces).
Notice that equation \eqref{minel} holds for every $z\in H^{1}_U(\Omega\setminus K)$.
\end{definition}

\medskip
\begin{figure}[tb]
\begin{center}
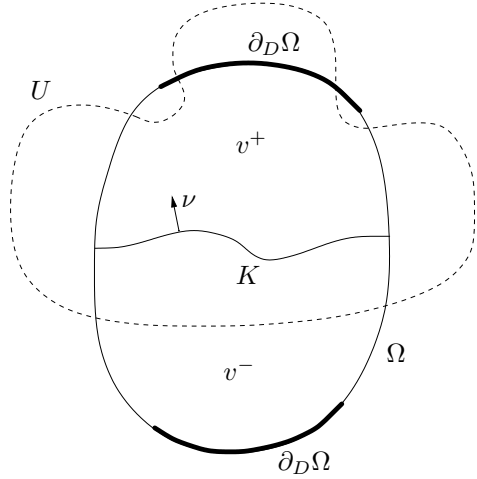
\end{center}
\caption{An admissible subdomain $U$ for a regular pair $(K,v)$ (see Definition~\ref{def:subdomain}). Notice that $U$ excludes the relative boundary of $\partial_D\Omega$.}
\end{figure}
\medskip

This paper deals with regular critical pairs $(\Gamma,u)$,
according to the following definition motivated by the formula for the first variation of the functional $F$
(see \eqref{varI} and Remark~\ref{rm:varI}).

\begin{definition} \label{def:critpair}
We say that a regular pair $(\Gamma,u)\in\spazioreg$ is a \emph{regular critical pair} for $F$ if the following conditions are satisfied:
\begin{itemize}
  \item[(i)] $\Gamma$ meets $\partial\Omega$ orthogonally,
  \item[(ii)] \emph{transmission condition}:
    \begin{equation} \label{EL1}
    H= |\nabla_\Gamma u^+|^2 - |\nabla_\Gamma u^-|^2 \qquad\text{on }\Gamma\cap\Omega,
    \end{equation}
  \item[(iii)] \emph{non-vanishing jump condition}: $|u^+-u^-|\geq c >0$ on $\Gamma$.
\end{itemize}
\end{definition}

\begin{remark} \label{rm:u}
The assumption of $C^\infty$-regularity of the curve $\Gamma$ is not so restrictive as it may appear:
indeed, as a consequence of the transmission condition \eqref{EL1} and of the fact that $u$ satisfies \eqref{minel},
$\Gamma$ is automatically analytical as soon as it is of class $C^{1,\alpha}$ (see \cite{KLM}).
Moreover, by \eqref{minel} $u$ is of class $C^\infty$ up to $\Gamma\cap\Omega$
and the traces $\nabla u^+$, $\nabla u^-$ of $\nabla u$ are well defined on both sides of $\Gamma$.
\end{remark}

Besides the notion of critical pair, which amounts to the vanishing of the \emph{first} variation of the functional,
we also introduce the concept of stability,
which is defined in terms of the positivity of the \emph{second} variation.
Its explicit expression at a regular critical pair $(\Gamma,u)$, which will be computed in Theorem~\ref{teo:varII},
motivates the definition of the quadratic form $\partial^2F((\Gamma,u);U):H^1(\Gamma\cap\Omega)\to\R$ given by
\begin{align} \label{d2f}
\partial^2F((\Gamma,u);U)[\vphi] :=
&-2\int_\Omega |\nabla \vf|^2\,dx
+\int_{\Gamma\cap\Omega} |\nabla_\Gamma\vphi|^2\,d\hu
+\int_{\Gamma\cap\Omega} H^2\vphi^2\,d\hu \nonumber\\
&-\int_{\nbordo} H_{\partial\Omega}\,\vphi^2\,d\huu
\end{align}
where $\vf\in\dir$ solves
\begin{equation} \label{vf}
\int_{\Omega} \nabla\vf \cdot \nabla z\,dx
+\int_{\Gamma\cap\Omega} \bigl[ z^+\div_\Gamma(\vphi\nabla_\Gamma u^+) - z^-\div_\Gamma(\vphi\nabla_\Gamma u^-) \bigr] \,d\hu =0
\end{equation}
for every $z\in\dir$.
Notice that the last integral in \eqref{d2f} in fact reduces to the sum $H_{\partial\Omega}(x_1)\vphi^2(x_1) + H_{\partial\Omega}(x_2)\vphi^2(x_2)$, where $x_1$ and $x_2$ are the intersections of $\Gamma$ with $\partial\Omega$.
The (nonlocal) dependence on $U$ is realized through the function $\vf$.

\begin{remark} \label{rm:hunmezzo}
The second integral in equation \eqref{vf} has to be intended
in the duality sense between $H^{-\frac12}(\Gamma\cap\Omega)$ and $H^{\frac12}(\Gamma\cap\Omega)$.
Indeed, by directly estimating the Gagliardo $H^{\frac12}$-seminorm one can check that the product
$\vphi\nabla_\Gamma u^\pm$ belongs to $H^\frac12(\Gamma\cap\Omega)$ as long as $\nabla_\Gamma u^\pm\in C^{0,\alpha}(\Gamma)$
for some $\alpha>\frac12$. In turn, the latter regularity property is guaranteed by Lemma~\ref{lemma:Neumannreg}.
\end{remark}

\begin{definition} \label{def:varIIpos}
We say that a regular critical pair $(\Gamma,u)$ (see Definition~\ref{def:critpair})
is \emph{strictly stable} in an admissible subdomain $U$ if
\begin{equation} \label{hppos}
\partial^2F((\Gamma,u);U)[\vphi]>0 \qquad\text{for every }\vphi\in H^1(\Gamma\cap\Omega)\setmeno\{0\}.
\end{equation}
\end{definition}

The aim of this paper is to discuss the relation between the notion of strict stability of a regular critical pair
and the one of local minimality.
It is easily seen that the positive semidefiniteness of the quadratic form $\partial^2F((\Gamma,u);U)$
is a necessary condition for local minimality in $U$ (see \cite[Theorem~3.15]{CMM}).
In the main result of the paper we prove that its \emph{strict positivity}
is in fact a \emph{sufficient} condition for a regular critical pair to be a local minimizer in the $L^1$-sense:

\begin{theorem} \label{teo:minSBV}
Let $(\Gamma,u)$ be a strictly stable regular critical pair in an admissible subdomain $U$,
according to Definition~\ref{def:varIIpos}.
Then $u$ is an isolated local minimizer for $F$ in $U$, in the sense that
there exists $\delta>0$ such that
\begin{equation}\label{locmin2}
F(\Gamma,u)< F(K,v)
\end{equation}
for every $(K,v)\in\spazio$ such that $v=u$ in $(\Omega\setminus U)\cup\partial_D\Omega$ and $0<\|u-v\|_{L^1(\Omega)}<\delta$.
\end{theorem}

\begin{remark}
In order to simplify the proofs and the notations we decided to state and prove the previous result
only in the simplified situation where $\Omega$ is connected and $\Gamma$ is a regular curve joining two points of $\partial\Omega$.
It is straightforward to check that Theorem~\ref{teo:minSBV} can be generalized
to the case where $\Gamma$ is a finite, disjoint union of curves of class $C^\infty$,
each one connecting two points of $\partial\Omega$
and meeting $\partial\Omega$ orthogonally.
\end{remark}

\begin{remark}
The non-vanishing jump condition (point (iii) of Definition~\ref{def:critpair})
is not a technical assumption and cannot be dropped: indeed, it is possible to construct examples
(see the Remark after Theorem~3.1 in \cite{DMMorMor})
satisfying all the assumptions of Theorem~\ref{teo:minSBV} except for this one,
for which the conclusion of the theorem does not hold.
In our strategy, this hypothesis is needed in order to deduce, in Proposition~\ref{prop:calibrazioni}, by applying the calibration constructed in \cite{Mor}, that the unique solution of the penalization problem \eqref{penalizzato} is $u$ itself , if $\beta$ is sufficiently large.
\end{remark}

We conclude with the following consequence of Theorem~\ref{teo:minSBV},
which states that given any family of equicoercive functionals $\F_\e$
which $\Gamma$-converge to the relaxed version of $F$ with respect to the $L^1$-topology,
we can approximate each strictly stable regular critical pair for $F$
by a sequence of local minimizers of the functionals $\F_\e$.
This follows from the abstract result observed in \cite[Theorem~4.1]{KS}.
There is a vast literature concerning the approximation of the Mumford-Shah functional in the sense of $\Gamma$-convergence (see, for instance, \cite{Bra}).

\begin{theorem}[link with $\Gamma$-convergence] \label{teo:gammaconv}
Let $(\Gamma,u)$ be a strictly stable regular critical pair in an admissible subdomain $U$.
Let $\F_\e:L^1(\Omega)\to\R\cup\{+\infty\}$ be a family of equicoercive and lower semi-continuous functionals which $\Gamma$-converge as $\e\to0$
to the relaxed functional (see the beginning of Section~\ref{sect:minSBV})
$$
\F(v):=
\left\{
  \begin{array}{ll}
    \int_{\Omega}|\nabla v|^2\,dx+\hu(S_v) & \hbox{if }v\in\sbv(\Omega),\, v=u \hbox{ on } (\Omega\setminus U)\cup\partial_D\Omega,\\
    +\infty & \hbox{otherwise in } L^1(\Omega)
  \end{array}
\right.
$$
with respect to the $L^1$-topology.
Then there exists $\e_0>0$ and a family $(u_\e)_{\e<\e_0}$ of local minimizers of $\F_\e$
such that $u_\e\to u$ in $L^1(\Omega)$ as $\e\to0$.
\end{theorem}

\end{section}


\begin{section}{Computation of the second variation} \label{sect:var2}

In this section we compute the second variation of the functional $F$.
To start with, we fix some notation:
for any one-parameter family of functions $(g_s)_{s\in\R}$
we denote the partial derivative with respect to the variable $s$
of the map $(s,x)\mapsto g_s(x)$, evaluated at $(t,x)$,
by $\dot{g}_t(x)$. We usually omit the subscript when $t=0$.
In the following, we fix a regular pair $(K,v)\in\spazioreg$ and an admissible subdomain $U$.

\begin{definition} \label{def:flow}
A flow $(\Phi_t)_t$ is said to be \emph{admissible} for $(K,v)$ in $U$ if it is generated by a vector field
$X\in C^2(\R^2;\R^2)$ such that $\supp X\subset\subset U\setminus\partial_D\Omega$ and $X\cdot\norm=0$ on $\partial\Omega$,
that is, $\Phi_t$ solves the equation $\dot{\Phi}_t=X\circ\Phi_t$, $\Phi_0=Id$.
\end{definition}

\begin{remark} \label{rm:flow}
The condition $X\cdot\norm=0$ guarantees that the trajectories of points in $\partial\Omega$ remain on $\partial\Omega$:
thus $\Phi_t(\ombar)=\ombar$ for every $t$.
Observe also that, since $\supp X\subset\subset U\setminus\partial_D\Omega$,
we have that $K_{\Phi_t}\subset U\setminus\partial_D\Omega$ for every $t$, where we set $\kft:=\Phi_t(K)$.
\end{remark}

Given an orientation preserving diffeomorphism $\Phi\in C^\infty(\ombar;\ombar)$
such that $\supp(\Phi-Id)\subset\subset U\setminus\partial_D\Omega$,
we define $v_\Phi$ as the unique solution in $H^{1}(\Omega\setminus\kf)$
(up to additive constants in the connected components of $\Omega\setminus\kf$ whose boundary does not contain $\partial_D\Omega$) to
\begin{equation} \label{uphi}
\left\{
  \begin{array}{ll}
    \displaystyle\int_{\Omega\setminus\kf} \nabla v_\Phi \cdot \nabla z=0 & \hbox{for every }z\in H^{1}_U(\Omega\setminus\kf),\\
    v_\Phi=v & \hbox{in } (\Omega\setminus U)\cup\partial_D\Omega.
  \end{array}
\right.
\end{equation}

\begin{definition} \label{def:variations}
Let $(\Phi_t)_t$ be an admissible flow for $(K,v)$ in $U$.
We define the \textit{first and second variations of $F$ at $(K,v)$ in $U$ along $(\Phi_t)_t$}
to be
$$
\frac{d}{dt} F ((\kft,v_{\Phi_t});U) |_{t=0}, \qquad
\frac{d^2}{dt^2} F ((\kft,v_{\Phi_t});U) |_{t=0}
$$
respectively, where $v_{\Phi_t}$ is defined as in \eqref{uphi} with $\Phi$ replaced by $\Phi_t$.
\end{definition}

Notice that this definition makes sense since the existence of the derivatives is guaranteed by the regularity result
proved in \cite[Proposition 8.1]{CMM}, which can be adapted to the present setting.
In particular, this result implies that the map $(t,x)\mapsto v_{\Phi_t}(x)$ is differentiable with respect to the variable $t$
and that $\dot{v}_{\Phi_t}\in H^1_U(\Omega\setminus\kft)$.
We set $\dot{v}:=\dot{v}_{\Phi_0}$

In the following theorem we compute explicitly the second variation of the functional $F$.
We stress that, comparing with the analogous result obtained in \cite[Theorem 3.6]{CMM},
we allow here the admissible variations to affect also the intersection of the discontinuity set $K$ with the boundary of $\Omega$,
while in the quoted paper only variations compactly supported in $\Omega$ were considered.
As a consequence, in the present situation boundary terms arise when integration by parts are performed:
in particular this happens for the derivatives of the surface term,
while the first and second variations of the volume term remain unchanged.
We refer also to \cite{SZ}, where a similar computation for the second variation of the surface area was carried out
taking into account boundary effects, in the case of a critical set
(the novelty here is that we will be able to get an expression of the second variation at a generic regular pair, not necessarily critical).

\begin{theorem} \label{teo:varII}
Let $(K,v)\in\spazioreg$ be a regular pair for $F$, let $U$ be an admissible subdomain,
and let $(\Phi_t)_t$ be an admissible flow in $U$ associated to a vector field $X$.
Then the function $\dot{v}$ belongs to $H^{1}_U(\Omega\setminus K)$ and satisfies the equation
\begin{equation} \label{upto}
\int_\Omega \nabla\dot{v} \cdot \nabla z\,dx
+\int_{K\cap\Omega} \Bigl[ \div_K\bigl((X\cdot\nu)\nabla_Kv^+\bigr)z^+
 -\div_K\bigl((X\cdot\nu)\nabla_Kv^-\bigr)z^- \Bigr] \,d\hu =0
\end{equation}
for every $z\in H^{1}_U(\Omega\setminus K)$.
Moreover, the second variation of $F$ at $(K,v)$ in $U$ along $(\Phi_t)_t$ is given by
\begin{align} \label{varIIgen}
\frac{d^2}{dt^2} &F ((\kft,v_{\Phi_t});U) |_{t=0} =
2\int_{K\cap\Omega} (\dot{v}^+\partial_\nu\dot{v}^+ - \dot{v}^-\partial_\nu\dot{v}^-)\,d\hu
+ \int_{K\cap\Omega} |\nabla_K(X\cdot\nu)|^2 \,d\hu \nonumber\\
&+\int_{K\cap\Omega} H^2 (X\cdot\nu)^2 \,d\hu
+\int_{K\cap\Omega} f (Z\cdot\nu - 2X^\|\cdot\nabla_K(X\cdot\nu) + D\nu[X^\|,X^\|] - H(X\cdot\nu)^2) \,d\hu \nonumber\\
&+\int_{\bordo} (f-H)(X\cdot\nu)(X\cdot\eta) \,d\huu
+\int_{\bordo} Z\cdot\eta \,d\huu\,,
\end{align}
where $f:=|\nabla_Kv^-|^2 - |\nabla_Kv^+|^2 + H$, $Z:=DX[X]$,
and we split the field $X$ in its tangential and normal components to $K$:
\begin{equation} \label{decomposition}
X=X^\| + (X\cdot\nu)\nu \quad\text{on }K.
\end{equation}
\end{theorem}

\begin{remark} \label{rm:upto}
As in \eqref{vf}, the second integral in equation \eqref{upto} has to be intended
in the duality sense between $H^{-\frac12}(K\cap\Omega)$ and $H^{\frac12}(K\cap\Omega)$
(see Remark~\ref{rm:hunmezzo}).
Integrations by parts yields
$$
- \int_\Omega |\nabla\dot{v}|^2\,dx =
\int_{K\cap\Omega} \bigl[ \dot{v}^+\partial_\nu\dot{v}^+ - \dot{v}^-\partial_\nu\dot{v}^- \bigr] \,d\hu.
$$
\end{remark}

Before proving Theorem~\ref{teo:varII},
we collect in the following lemma some auxiliary identities which will be used
in the computation of the second variation.

\begin{lemma} \label{lemma:uguaglianze2}
The following identities hold:
\begin{enumerate}
\smallskip
    \item[(a)] $\dot{\nu} = -(D_KX)^T[\nu] - D_K\nu[X] = -\nabla_K(X\cdot\nu)$ on $K$;
\medskip
    \item[(b)] $\frac{\partial}{\partial t} (\etaft\circ\Phi_t) |_{t=0} = (D_KX)^T[\nu,\eta]\nu$ on $\bordo$;
\medskip
    \item[(c)] $(X\cdot\nu)\dot{\nu}\cdot\eta + X\cdot\frac{\partial}{\partial t} (\etaft\circ\Phi_t) |_{t=0} = -H(X\cdot\nu)(X\cdot\eta)$ on $\bordo$;
\medskip
    \item[(d)] $DX[X,\norm] + D\norm[X,X] =0$ on $\bordo$.
\smallskip
\end{enumerate}
\end{lemma}

\begin{proof}
Equality (a) is proved in \cite[Lemma 3.8, (f)]{CMM}.
To prove (b), we set $v_t := D\Phi_t[\eta]$ and recalling \eqref{etaf} we have
\begin{align*}
\frac{\partial}{\partial t} (\etaft\circ\Phi_t) |_{t=0}
& = \frac{\partial}{\partial t} \biggl(\frac{v_t}{|v_t|}\biggr) \bigg|_{t=0}
  =  \dot{v} - (\dot{v}\cdot\eta)\eta \\
& = DX[\eta] - DX[\eta,\eta]\eta
  = DX[\eta,\nu]\nu,
\end{align*}
which is (b).
We obtain (c) by combining (a) and (b):
\begin{align*}
(X\cdot\nu)\dot{\nu}\cdot\eta + X\cdot\frac{\partial}{\partial t} (\etaft\circ\Phi_t) |_{t=0}
= -(X\cdot\nu)D_K\nu[X,\eta]
= - H(X\cdot\nu)(X\cdot\eta),
\end{align*}
where the last equality follows by writing $X=(X\cdot\nu)\nu+(X\cdot\eta)\eta$ and observing that $D_K\nu[\nu]=0$.
Equation (d) follows by differentiating with respect to $t$ at $t=0$ the identity
$$
(X\circ\Phi_t)\cdot(\norm\circ\Phi_t)=0,
$$
which holds on $\bordo$.
\end{proof}

\begin{proof}[Proof of Theorem~\ref{teo:varII}]
We split the proof of the theorem into three steps.

\smallskip
\noindent
{\it Step 1. Derivation of the equation solved by $\dot{v}$.}
As already observed, the result contained in \cite[Proposition~8.1]{CMM} guarantees that $\dot{v}\in H^1_U(\Omega\setminus K)$.
Given any test function $z\in H^1_U(\Omega\setminus K)$ with $\supp z\cap K=\emptyset$,
for $t$ small enough we have $\supp z \subset\Omega\setminus K_{\Phi_t}$,
and in particular $z\in H^1_U(\Omega\setminus\kft)$.
Hence by \eqref{uphi} we deduce
$$
\int_{\Omega}\nabla v_{\Phi_t}\cdot\nabla z\,dx=0,
$$
so that differentiating with respect to $t$ at $t=0$ we obtain that $\dot{v}$ is harmonic in $(\Omega\cap U)\setminus K$
and $\nabla\dot{v}\cdot\norm=0$ on $\partial\Omega\cap U$.
In addition, it is shown in Step 1 of the proof of \cite[Theorem 3.6]{CMM} that
\begin{equation*}
\partial_\nu\dot{v}^\pm = \div_K((X\cdot\nu)\nabla_Kv^\pm) \qquad\text{on }K\cap\Omega.
\end{equation*}
By this expression we have that $\partial_\nu\dot{v}^\pm\in H^{-\frac12}(K\cap\Omega)$
(see Remark~\ref{rm:hunmezzo}), and hence
the previous conditions are equivalent to \eqref{upto} by integration by parts.

\smallskip
\noindent
{\it Step 2. Computation of the first variation.}
The same computation carried out in Step 2 of the proof of \cite[Theorem 3.6]{CMM} leads to
$$
\frac{d}{dt}\int_{\Omega}|\nabla v_{\Phi_t}|^2\,dx
= \int_{\Omega} \div\bigl( |\nabla v_{\Phi_t}|^2\ncampot \bigr) \,dy\,.
$$
Hence, applying the divergence theorem we obtain
\begin{align*}
\frac{d}{dt}\int_{\Omega}|\nabla v_{\Phi_t}|^2\,dx
&= \int_{\partial\Omega} |\nabla v_{\Phi_t}|^2 (\ncampot\cdot\norm)\,d\hu
+ \int_{\kft\cap\Omega} \bigl( |\nabla v_{\Phi_t}^-|^2 - |\nabla v_{\Phi_t}^+|^2 \bigr)(\ncampot\cdot\nuft)\,d\hu\\
&= \int_{\kft\cap\Omega} \bigl( |\nabla_{\kft} v_{\Phi_t}^-|^2 - |\nabla_{\kft} v_{\Phi_t}^+|^2 \bigr)(\ncampot\cdot\nuft)\,d\hu \,
\end{align*}
where to deduce the last equality we used $X\cdot\norm=0$
and the fact that $\partial_{\nuft}v_{\Phi_t}^\pm$ vanishes on $\kft$.
Concerning the surface term,
we start from the well known formula for the first variation of the area functional
(see, for instance, \cite[Chapter~2, Section~9]{Sim})
and we use the divergence theorem on $\kft\cap\Omega$, to obtain
\begin{align*}
\frac{d}{dt}\hu(\kft\cap\Omega)
&= \int_{\kft\cap\Omega} \div_{\kft}\ncampot\,d\hu \\
&= \int_{\kft\cap\Omega} H_{\flusso}(\ncampot\cdot\nuft)\,d\hu
+ \int_{\bordoft}\ncampot\cdot\etaft\,d\huu\,,
\end{align*}
where we recall that $H_{\flusso}$ stands for the curvature of $\kft$.
Thus we can conclude that
\begin{equation} \label{varIt}
\frac{d}{dt} F ((\kft,v_{\Phi_t});U) =
\int_{\kft\cap\Omega} f_t (\ncampot\cdot\nuft)\,d\hu
+ \int_{\bordoft}\ncampot\cdot\etaft\,d\huu\,,
\end{equation}
where $f_t:=|\nabla_{\kft} v_{\Phi_t}^-|^2 - |\nabla_{\kft} v_{\Phi_t}^+|^2+H_\flusso$.
In particular, evaluating \eqref{varIt} at $t=0$ we obtain
\begin{equation}  \label{varI}
\frac{d}{dt} F ((\kft,v_{\Phi_t});U) |_{t=0} =
\int_{K\cap\Omega} f(X\cdot\nu)\,d\hu
+ \int_{K\cap\partial\Omega}X\cdot\eta\,d\huu\,.
\end{equation}

\smallskip
\noindent
{\it Step 3. Computation of the second variation.}
We have to differentiate again \eqref{varIt} at $t=0$.
By a change of variables we have
\begin{align} \label{contivarII0}
\frac{d^2}{dt^2} &F ((\kft,v_{\Phi_t});U) |_{t=0}
= \int_{K\cap\Omega} \frac{\partial}{\partial t}(f_t\circ\Phi_t)|_{t=0}(X\cdot\nu)\,d\hu \nonumber\\
& + \int_{K\cap\Omega} f \frac{\partial}{\partial t} \bigl( \dot{\Phi}_t\cdot(\nuft\circ\Phi_t)J_{\Phi_t} \bigr)\big|_{t=0}\,d\hu
+ \frac{d}{d t} \biggl( \int_{\bordoft}\ncampot\cdot\etaft\,d\huu \biggr) \bigg|_{t=0} \nonumber\\
&=:I_1+I_2+I_3\,.
\end{align}
The first integral $I_1$ is equal to
\begin{align} \label{contivarII1}
I_1 &= \int_{K\cap\Omega} \dot{f}(X\cdot\nu)\,d\hu
+ \int_{K\cap\Omega} (\nabla f\cdot\nu)(X\cdot\nu)^2\,d\hu
+ \int_{K\cap\Omega} (\nabla_Kf\cdot X^\|)(X\cdot\nu)\,d\hu \,,
\end{align}
while using \cite[Lemma~3.8,~(g)]{CMM} we have
\begin{align} \label{contivarII1bis}
I_2 = \int_{K\cap\Omega} f\div_K((X\cdot\nu)X)\,d\hu
+ \int_{K\cap\Omega} f \bigl(Z\cdot\nu - 2X^\|\cdot\nabla_K(X\cdot\nu) + D\nu[X^\|,X^\|] \bigr)\,d\hu \,.
\end{align}
Applying the divergence formula on $K\cap\Omega$ we obtain
\begin{align} \label{contivarII2}
\int_{K\cap\Omega} (\nabla_Kf\cdot X^\|)&(X\cdot\nu)\,d\hu
+ \int_{K\cap\Omega} f\div_K((X\cdot\nu)X)\,d\hu \nonumber\\
&= \int_{K\cap\Omega} fH(X\cdot\nu)^2\,d\hu + \int_{\bordo}f(X\cdot\nu)(X\cdot\eta)\,d\huu \,,
\end{align}
while using \cite[formula~(3.17)]{CMM} we get
\begin{align} \label{contivarII3}
\int_{K\cap\Omega} (\nabla f\cdot\nu)(X\cdot\nu)^2\,d\hu
& = \int_{K\cap\Omega} (H^2-2fH)(X\cdot\nu)^2 \,d\hu.
\end{align}
Differentiating $f_t$ with respect to $t$ we obtain
\begin{align} \label{contivarII3bis}
\int_{K\cap\Omega} \dot{f}&(X\cdot\nu)\,d\hu
= \int_{K\cap\Omega} (2\nabla_Kv^-\cdot\nabla_K\dot{v}^- - 2\nabla_Kv^+\cdot\nabla_K\dot{v}^+ + \dot{H}) (X\cdot\nu)\,d\hu,
\end{align}
and an integration by parts yields
\begin{align} \label{contivarII4}
2\int_{K\cap\Omega} &(\nabla_Kv^\pm\cdot\nabla_K\dot{v}^\pm)(X\cdot\nu)\,d\hu \nonumber\\
&= -2\int_{K\cap\Omega} \dot{v}^\pm\div_K((X\cdot\nu)\nabla_Kv^\pm)\,d\hu
+ 2\int_{\bordo}\dot{v}^\pm(X\cdot\nu)(\nabla_Kv^\pm\cdot\eta)\,d\huu \nonumber\\
&= -2\int_{K\cap\Omega} \dot{v}^\pm\partial_\nu\dot{v}^\pm\,d\hu,
\end{align}
where the last equality follows by \eqref{upto} and by observing that
$\nabla v^\pm$ vanishes on $\bordo$,
as $v$ satisfies homogeneous Neumann boundary conditions on $K$ and on $\partial\Omega$
($\nabla v$ is regular up to $K\cap\partial\Omega$ by Lemma~\ref{lemma:Neumannreg}).
Since $\partial_\nu\dot{\nu}\cdot\nu=-\dot{\nu}\cdot\partial_\nu\nu=0$,
we have $\div\dot{\nu}=\div_K\dot{\nu}$
and in turn $\dot{H}=\div_K\dot{\nu}$.
Hence, integrating by parts and using (a) of Lemma~\ref{lemma:uguaglianze2}, we deduce
\begin{align} \label{contivarII5}
\int_{K\cap\Omega} \dot{H}(X\cdot\nu)\,d\hu
&= \int_{K\cap\Omega} \div_K\dot{\nu}(X\cdot\nu)\,d\hu \nonumber\\
&= -\int_{K\cap\Omega} \dot{\nu}\cdot\nabla_K(X\cdot\nu)\,d\hu
+ \int_{\bordo} (X\cdot\nu)(\dot{\nu}\cdot\eta)\,d\huu \nonumber\\
&= \int_{K\cap\Omega} |\nabla_K(X\cdot\nu)|^2\,d\hu
+ \int_{\bordo} (X\cdot\nu)(\dot{\nu}\cdot\eta)\,d\huu.
\end{align}
We finally compute $I_3$:
\begin{align} \label{contivarII6}
I_3
&=\frac{d}{dt} \biggl( \int_{\bordoft} \ncampot\cdot\etaft \,d\huu \biggr) \bigg|_{t=0}
= \sum_{x\in\bordo} \frac{\partial}{\partial t} \Bigl( \ncampot(\Phi_t(x))\cdot \etaft(\Phi_t(x)) \Bigr) \Big|_{t=0} \nonumber\\
&= \sum_{x\in\bordo} Z(x)\cdot\eta(x) + \sum_{x\in\bordo}  X(x) \cdot \frac{\partial}{\partial t}(\etaft\circ\Phi_t(x))|_{t=0}\,.
\end{align}
Collecting \eqref{contivarII0}--\eqref{contivarII6},
and using equality (c) of Lemma~\ref{lemma:uguaglianze2}, we finally obtain \eqref{varIIgen}.
\end{proof}

\begin{remark} \label{rm:varIIt}
We observe that we can easily obtain an expression for the second variation of the functional $F$ at a generic $t$.
Indeed, by exploiting the property $\Phi_{t+s}=\Phi_t\circ\Phi_s$ of the flow, we have
$$
\frac{d^2}{dh^2} F ((K_{\Phi_h},v_{\Phi_h});U) |_{h=t}
= \frac{d^2}{ds^2} F ((\Phi_{t+s}(K),v_{\Phi_{t+s}});U) |_{s=0}
= \frac{d^2}{ds^2} F (\Phi_s(K_{\Phi_t}),(v_{\Phi_t})_{\Phi_s}) |_{s=0},
$$
and we can directly apply Theorem~\ref{teo:varII} to the regular pair $(\kft,v_{\Phi_t})$.
\end{remark}

\begin{remark} \label{rm:varI}
The formula \eqref{varI} for the first variation of $F$ motivates the definition of critical pair
(see Definition~\ref{def:critpair}).
Indeed, assuming that \eqref{varI} vanishes for each vector field $X$ which is tangent to $\partial\Omega$,
we first obtain that $f=0$ on $K\cap\Omega$ by considering arbitrary vector fields with $\supp X\subset\subset\Omega$.
Then, using this information and dropping the requirement on the support of $X$,
we deduce the orthogonality of $K$ and $\partial\Omega$.
\end{remark}

\begin{corollary} \label{cor:varII}
Assume that $(\Gamma,u)$ is a regular critical pair. Then
\begin{align} \label{varII}
\frac{d^2}{dt^2} F ((\Gamma_{\Phi_t},u_{\Phi_t});&U) |_{t=0} =
- 2\int_\Omega |\nabla\dot{u}|^2 \,dx
+ \int_{\Gamma\cap\Omega} |\nabla_\Gamma(X\cdot\nu)|^2 \,d\hu \nonumber\\
& +\int_{\Gamma\cap\Omega} H^2 (X\cdot\nu)^2 \,d\hu
- \int_{\nbordo} H_{\partial\Omega}(X\cdot\nu)^2 \,d\huu\,,
\end{align}
where $H_{\partial\Omega}:=\div\norm$ denotes the curvature of $\partial\Omega$.
\end{corollary}

\begin{proof}
The first integral in \eqref{varIIgen} can be rewritten as $-2\int_\Omega|\nabla\dot{u}|^2\,dx$
thanks to \eqref{upto} (see Remark~\ref{rm:upto}).
To obtain the expression in \eqref{varII} it is now sufficient to observe that
at a critical pair we have $f=0$ on $K\cap\Omega$, $X\cdot\eta=X\cdot\norm=0$ on $\bordo$, and
$$
Z\cdot\eta = DX[X,\norm] = -D\norm[X,X] = -(X\cdot\nu)^2D\norm[\nu,\nu] = - H_{\partial\Omega}(X\cdot\nu)^2
$$
on $\bordo$ by (d) of Lemma~\ref{lemma:uguaglianze2}.
\end{proof}

\subsection{The second order condition}

In the following we assume that $(\Gamma,u)$ is a regular critical pair and $U$ is an admissible subdomain.
Notice that the expression of the second variation of $F$ at $(\Gamma,u)$ proved in Corollary~\ref{cor:varII}
motivates the definition of the quadratic form \eqref{d2f}
and the notion of strict stability that we introduced in Definition~\ref{def:varIIpos}.

Following the approach of \cite{CMM},
we start paving the way for the main result by proving two equivalent formulations of condition \eqref{hppos},
one in terms of the first eigenvalue of a suitable compact linear operator defined on $H^1(\Gamma\cap\Omega)$
and the other in terms of a dual minimum problem.
Let us start by introducing the following bilinear form on $H^1(\Gamma\cap\Omega)$:
\begin{equation} \label{prscal}
(\vphi,\psi)_\sim :=
\int_{\Gamma\cap\Omega}\nabla_\Gamma\vphi\cdot\nabla_\Gamma\psi\,d\hu
+\int_{\Gamma\cap\Omega}H^2\,\vphi\,\psi\,d\hu
-\int_{\nbordo}H_{\partial\Omega}\,\vphi\,\psi\,d\huu
\end{equation}
for every $\vphi,\psi\in H^1(\Gamma\cap\Omega)$.
The proof of the following proposition can be obtained by simply adapting \cite[Proposition~4.2]{CMM}
to our slightly different situation.

\begin{proposition} \label{prop:ps}
Assume that
\begin{equation} \label{hpps}
(\vphi,\vphi)_\sim>0 \qquad\text{for every }\vphi\in H^1(\Gamma\cap\Omega)\setmeno\{0\}.
\end{equation}
Then $(\cdot,\cdot)_\sim$ is a scalar product
which defines an equivalent norm on $H^1(\Gamma\cap\Omega)$,
denoted by $\|\cdot\|_\sim$.
\end{proposition}

%

The announced equivalent formulations of condition to \eqref{hppos} are stated in the following proposition.
Also for this proof we refer the reader to Proposition~4.3, Theorem~4.6 and Theorem~4.10 in \cite{CMM},
which can be directly adapted with the natural modifications,
taking into account Remark~\ref{rm:hunmezzo}.

\begin{proposition} \label{prop:T}
The following statements are equivalent.
\begin{enumerate}
\item[(i)] Condition \eqref{hppos} is satisfied.
\item[(ii)] Condition \eqref{hpps} holds, and the monotone, compact, self-adjoint operator $T:H^1(\Gamma\cap\Omega)\to H^1(\Gamma\cap\Omega)$ defined by duality as
\begin{equation} \label{T}
(T\vphi,\psi)_\sim =
-2 \int_{\Gamma\cap\Omega} \bigl[ \vf^+\div_\Gamma(\psi\nabla_\Gamma u^+) - \vf^-\div_\Gamma(\psi\nabla_\Gamma u^-) \bigr] \,d\hu
\end{equation}
for every $\vphi,\psi\in H^1(\Gamma\cap\Omega)$
(where $\vf$ is defined in \eqref{vf}), satisfies
\begin{equation} \label{lamda1}
\lambda_1(U):=\max_{\|\vphi\|_\sim=1} (T\vphi,\vphi)_\sim <1
\end{equation}
(where the dependence on $U$ is realized through the function $\vf$).
\item[(iii)] Condition \eqref{hpps} holds, and defining, for $v\in\dir$, $\Phi_v$ as the unique solution in $H^1(\Gamma\cap\Omega)$ to
\begin{equation*}
(\Phi_v,\psi)_\sim =
-2 \int_{\Gamma\cap\Omega} \bigl[ v^+\div_\Gamma(\psi\nabla_\Gamma u^+) - v^-\div_\Gamma(\psi\nabla_\Gamma u^-) \bigr] \,d\hu
\end{equation*}
for every $\psi\in H^1(\Gamma\cap\Omega)$, one has
\begin{equation} \label{mu1}
\mu(U):=\min\Bigl\{2\int_\Omega|\nabla v|^2\,dx : v\in\dir, \, \|\Phi_v\|_\sim=1 \Bigr\}>1.
\end{equation}
\end{enumerate}
We will omit the dependence on $U$ for $\lambda_1$ and $\mu$ where there is no risk of ambiguity.
\end{proposition}

\begin{remark} \label{rm:T}
Notice that if condition \eqref{hpps} is satisfied,
then by Proposition~\ref{prop:ps} and by the Riesz Theorem the operator $T$ is well defined.
By \eqref{vf} we immediately have
$$
(T\vphi,\psi)_\sim = 2\int_\Omega \nabla\vf\cdot\nabla v_{\psi}\,dx.
$$
Moreover comparing with \eqref{d2f} we see that
$$
\partial^2F((\Gamma,u);U)[\vphi] = -(T\vphi,\vphi)_\sim + \|\vphi\|^2_\sim.
$$
\end{remark}

%
%

\begin{corollary}
Assume \eqref{hppos}. Then there exists a constant $C>0$ such that
$$
\partial^2F((\Gamma,u);U)[\vphi]\geq C\|\vphi\|^2_{H^1(\Gamma\cap\Omega)} \qquad\text{for every } \vphi\in H^1(\Gamma\cap\Omega).
$$
\end{corollary}

\begin{proof}
By Remark~\ref{rm:T} and \eqref{lamda1}
$$
\partial^2F((\Gamma,u);U)[\vphi] = \|\vphi\|^2_\sim - (T\vphi,\vphi)_\sim \geq (1-\lambda_1)\|\vphi\|^2_\sim,
$$
hence the conclusion follows by Proposition~\ref{prop:T} and Proposition~\ref{prop:ps}.
\end{proof}

From the definition in \eqref{mu1} it is clear that $\mu$ depends monotonically on the domain $U$.
This is made explicit by the following corollary.

\begin{corollary} \label{cor:largerdomain0}
Let $U_1$, $U_2$ be admissible subdomains for $(\Gamma,u)$, with $U_1\subset U_2$.
Then $\mu(U_1)\geq\mu(U_2)$.
In particular, if condition \eqref{hppos} is satisfied in $U_2$, then it also holds in $U_1$.
\end{corollary}

\begin{corollary} \label{cor:largerdomain}
Assume that condition \eqref{hppos} holds in $U$.
Let $U_n$ be a decreasing sequence of admissible subdomains for $(\Gamma,u)$ such that
$U$ is the interior part of $\bigcap_nU_n$.
Then \eqref{hppos} holds in $U_n$, if $n$ is sufficiently large.
\end{corollary}

\begin{proof}
In view of \eqref{mu1} it is sufficient to show that $\lim_n\mu(U_n)\geq\mu(U)$.
Let $v_n\in H^1_{U_n}(\Omega\setminus\Gamma)$ be a solution to \eqref{mu1} with $U$ replaced by $U_n$.
Then $v_n\in H^1_{U_1}(\Omega\setminus\Gamma)$
and $2\int_{\Omega}|\nabla v_n|^2\,dx=\mu(U_n)\leq\mu(U)$,
where the inequality follows from Corollary~\ref{cor:largerdomain0}.
Hence, up to subsequences, $v_n\wto v\in H^1_{U_1}(\Omega\setminus\Gamma)$.
Moreover, $v=0$ a.e. in $U_1\setminus U$, so that $v\in H^1_U(\Omega\setminus\Gamma)$
and $v$ is admissible in problem \eqref{mu1} (by the compactness of the map $v\mapsto\Phi_v$): we conclude that
$$
\lim_{n\to\infty}\mu(U_n) = \lim_{n\to\infty} 2\int_{\Omega}|\nabla v_n|^2\,dx
\geq 2\int_{\Omega}|\nabla v|^2\,dx \geq \mu(U),
$$
as claimed.
\end{proof}

\end{section}


\begin{section}{Local $W^{2,\infty}$-minimality} \label{sect:minW}

In this section, as a first step toward the proof of Theorem~\ref{teo:minSBV},
we show how the strategy developed in \cite{CMM} can be adapted to the present setting
in order to prove that the positiveness condition \eqref{hppos}
is sufficient for a regular critical pair to be a local minimizer with respect to variations of class $W^{2,\infty}$
of the discontinuity set.
For the rest of the section $(\Gamma,u)$ will be a fixed strictly stable regular critical pair in an admissible subdomain $U$.
For $\eta>0$, we denote by
$$
\mathcal{N}_{\eta}(\Gamma) := \{x\in\R^2 : \dist(x,\Gamma)<\eta\}
$$
the $\eta$-tubular neighborhood of $\Gamma$.

In order to give a proper notion of sets which are close to $\Gamma$ in the $W^{2,\infty}$-sense,
we now introduce a suitable flow in $U$ whose trajectories intersect $\Gamma$ orthogonally.
To this aim, we start by fixing $\eta_0>0$ such that $\mathcal{N}_{\eta_0}(\Gamma)\subset\subset U\setminus\partial_D\Omega$,
and a vector field $X\in C^2(\R^2;\R^2)$ such that $\supp X\subset\subset U\setminus\partial_D\Omega$,
$X=\nu$ on $\Gamma$, $X\cdot\norm=0$ on $\partial\Omega$, and $|X|=1$ in $\mathcal{N}_{\eta_0}(\Gamma)$.
We denote by $\Psi:\R\times\ombar\to\ombar$ the flow generated by $X$:
$$
\frac{\partial}{\partial t}\Psi(t,x)=X(\Psi(t,x)),
\qquad
\Psi(0,x)=x.
$$
Observe that (by taking a smaller $\eta_0$ if necessary)
for every $y\in\mathcal{N}_{\eta_0}(\Gamma)$ are uniquely determined two points
$\pi(y)\in\Gamma$ and $\tau(y)\in\R$ such that $y=\Psi(\tau(y),\pi(y))$.
The existence of the maps $\pi$ and $\tau$,
as well as the fact that they are of class $C^2$,
is guaranteed by the Implicit Function Theorem.

We define, for $\delta>0$, the following class of functions:
\begin{align*}
\mathcal{D}_\delta :=
\bigl\{ \psi\in C^{2}(\Gamma) :  \|\psi\|_{C^{2}(\Gamma)}<\delta \}.
\end{align*}
We can extend each function $\psi\in\mathcal{D}_\delta$ to $\mathcal{N}_{\eta_0}(\Gamma)$
by setting $\psi(y):=\psi(\pi(y))$, in such a way that $\psi$ is constant along the trajectories of the flow $\Psi$.
We associate with $\psi$ the diffeomorphism $\Phi^\psi(x):=\Psi(\psi(x),x)$, and we remark that
\begin{equation} \label{eq:stimapsi}
\|\Phi^\psi-Id\|_{C^{2}(\Gamma)} \leq C \|\psi\|_{C^{2}(\Gamma)}
\end{equation}
for some constant $C$ independent of $\psi\in\mathcal{D}_\delta$.
Finally, we define the set
\begin{equation} \label{eq:gammapsi}
\Gamma_\psi := \Phi^\psi(\Gamma) = \{\Psi(\psi(x),x) : x\in\Gamma\}\,,
\end{equation}
and the function $u_\psi:= u_{\Phi^\psi}$ as the unique solution in $H^1(\Omega\setminus\Gamma_\psi)$ to
$$
\int_{\Omega\setminus\Gamma_\psi}\nabla u_\psi\cdot\nabla z = 0
\qquad\text{for every }z\in H^1_U(\Omega\setminus\Gamma_\psi)
$$
with $u_\psi=u$ in $(\Omega\setminus U)\cup\partial_D\Omega$.
We will also denote by $\nu_\psi:=\nu_{\Phi^\psi}$ and $\eta_\psi:=\eta_{\Phi^\psi}$
the vectors defined on $\Gamma_\psi$ and $\Gamma_\psi\cap\partial\Omega$ by \eqref{nuf} and \eqref{etaf} respectively,
and by $H_\psi:=\div_{\Gamma_\psi}\nu_\psi$ the curvature of $\Gamma_\psi$.

\begin{remark} \label{rm:Neumannreg}
For $\psi\in\mathcal{D}_{\delta}$, the function $u_\psi$ is a weak solution to the Neumann problem
$$
\left\{
  \begin{array}{ll}
    \Delta u_\psi=0 & \hbox{in } (\Omega\cap U) \setminus \Gamma_\psi,\\
    \partial_{\nu_\psi} u_\psi=0 & \hbox{on } \Gamma_\psi\cap\Omega,\\
    \partial_{\norm} u_\psi =0 & \hbox{on } (\partial\Omega\cap U)\setminus\partial_D\Omega,
  \end{array}
\right.
$$
and the sets $\Gamma_\psi$ are uniformly bounded in $C^2$, by \eqref{eq:stimapsi}.
Hence, by classical results and by using Lemma~\ref{lemma:Neumannreg}
to deal with the regularity in a neighborhood of the boundary $\Gamma_\psi\cap\partial\Omega$,
we obtain that the functions $u_\psi^\pm$ are of class $C^{1,\gamma}$ up to $\Gamma_\psi\cap\ombar$, for some $\gamma\in(\frac12,1)$,
with $C^{1,\gamma}$-norm uniformly bounded with respect to $\psi\in\mathcal{D}_{\delta}$.
More precisely,
$$
\sup_{\psi\in\mathcal{D}_\delta} \|\nabla_\Gamma(u_\psi^\pm\circ\Phi^\psi)\|_{C^{0,\gamma}(\Gamma\cap\ombar;\R^2)} <+\infty \,,
$$
and, as an application of Ascoli--Arzel\`{a} Theorem, we also have
$$
\sup_{\psi\in\mathcal{D}_\delta} \|\nabla_\Gamma (u_\psi^\pm\circ\Phi^\psi) - \nabla_\Gamma u^\pm\|_{C^{0,\alpha}(\Gamma\cap\ombar;\R^2)} \to0
$$
for every $\alpha\in(0,\gamma)$, as $\delta\to0$.
\end{remark}

The main result of this section is the following.

\begin{theorem} \label{teo:minW}
Let $(\Gamma,u)$ be a strictly stable regular critical pair in an admissible subdomain $U$,
according to Definition~\ref{def:varIIpos}.
Then $(\Gamma,u)$ is an isolated local $W^{2,\infty}$-minimizer for $F$ in $U$, in the sense that
there exist $\delta>0$ and $C>0$ such that
$$
F(\Gamma_\psi,v) \geq F(\Gamma,u) + C\,\|\psi\|^2_{H^{1}(\Gamma\cap\Omega)}
$$
for every $\psi\in W^{2,\infty}(\Gamma\cap\Omega)$ such that $\|\psi\|_{W^{2,\infty}(\Gamma\cap\Omega)}<\delta$,
and for every $v\in H^1(\Omega\setminus \Gamma_\psi)$ with $v=u$ in $(\Omega\setminus U)\cup\partial_D\Omega$
(where the set $\Gamma_\psi$ is defined in \eqref{eq:gammapsi}).
\end{theorem}

The remaining part of this section is entirely devoted to the proof of Theorem~\ref{teo:minW}.
We start by fixing $\delta_0>0$ such that $\Gamma_\psi\subset\mathcal{N}_{\eta_0}(\Gamma)$ for every $\psi\in\mathcal{D}_{\delta_0}$,
where $\mathcal{N}_{\eta_0}(\Gamma)$ is the tubular neighborhood of $\Gamma$ fixed at the beginning of this section.
Our first task is to associate, with every $\psi\in\mathcal{D}_{\delta_0}$,
an admissible flow $(\Phi_t)_t$ connecting $\Gamma$ to $\Gamma_\psi$:
this can be easily done by setting
\begin{equation} \label{eq:flusso}
\Phi_t(x):=\Psi(t\psi(x),x).
\end{equation}
The flow $\Phi_t$ is admissible in $U$ (according to Definition~\ref{def:flow}),
as it is generated by the vector field
\begin{equation} \label{eq:fieldX}
\nfield := \psi X,
\end{equation}
where $X$ is defined at the beginning of this section.
Moreover it satisfies $\Phi_1(\Gamma)=\Gamma_\psi$, and
\begin{equation} \label{stimaflusso}
\|\Phi_t-Id\|_{C^{2}(\Gamma)} \leq C \, \|\psi\|_{C^{2}(\Gamma)}
\end{equation}
for every $t\in[0,1]$, where $C$ is a positive constant independent of $\psi\in\mathcal{D}_{\delta_0}$.
We also introduce the vector field
\begin{equation} \label{eq:fieldZ}
Z_\psi := D\nfield[\nfield] =\psi^2 DX[X]
\end{equation}
(the last equality follows by a direct computation, by observing that $\nabla\psi\cdot X=0$ since $\psi$ is constant along the trajectories of the flow generated by $X$).
Notice that by \eqref{eq:fieldX} and \eqref{eq:fieldZ} we immediately have the estimates
\begin{equation} \label{stima0ter}
|\nfield| \leq |\psi|,
\qquad
|Z_\psi| \leq C \, |\psi|^2
\qquad\text{in }\mathcal{N}_{\eta_0}(\Gamma),
\end{equation}
where $C$ is a positive constant independent of $\psi$.
In the following lemma we collect some technical estimates
concerning the above construction
that will be used in the proof of the main result of this section.

\begin{lemma} \label{lemma:stimecampo}
Given $\e>0$, there exists $\delta(\e)>0$ such that for every $\psi\in\mathcal{D}_{\delta(\e)}$
the following estimates hold:
\begin{enumerate}
\item[(a)] $\frac12\|\psi\|^2_{H^1(\Gamma\cap\Omega)} \leq \|\nfield\cdot\nu_\psi\|^2_{H^1(\Gamma_\psi\cap\Omega)} \leq 2\|\psi\|^2_{H^1(\Gamma\cap\Omega)}$;
\medskip
\item[(b)] $|\nfield\cdot\eta_\psi| \leq \e \, |\psi|$ on $\Gamma_\psi\cap\partial\Omega$.
\medskip
\item[(c)] $\frac12 \, \|\psi\|^2_{H^1(\Gamma\cap\Omega)} \leq \|\psi\|^2_{H^1(\Gamma_{\psi}\cap\Omega)} \leq 2\, \|\psi\|^2_{H^1(\Gamma\cap\Omega)}$.
\end{enumerate}
\end{lemma}

\begin{proof}
To prove (a), we first note that given $\sigma>0$ we can find $\delta(\sigma)\in(0,\delta_0)$ such that
for every $\psi\in\mathcal{D}_{\delta(\sigma)}$ we have on $\Gamma_\psi$
\begin{equation} \label{conti1}
\nu_{\psi} = \nu\circ\Phi_1^{-1} + \tilde{\nu}
\quad\text{with}
\quad \|\tilde{\nu}\|_{C^{1}(\Gamma_\psi)} \leq \sigma
\end{equation}
and
\begin{equation} \label{conti2}
\|X-X\circ\Phi_1^{-1}\|_{C^{1}(\Gamma_\psi)} \leq \sigma
\end{equation}
(where $\Phi_1=\Phi^\psi$, by \eqref{eq:flusso}).
Hence on $\Gamma_\psi$
\begin{align*}
\nfield\cdot\nu_\psi = \psi X\cdot\nu_\psi
= \psi \bigl( (X\cdot\nu)\circ\Phi_1^{-1} + (X-X\circ\Phi_1^{-1})\cdot\nu\circ\Phi_1^{-1} + X\cdot\tilde{\nu} \bigr)
= : \psi(1+R_1)
\end{align*}
(where we used the fact that $(X\cdot\nu)\circ\Phi_1^{-1}=1$), and
\begin{align*}
\nabla_{\Gamma_\psi} (\nfield\cdot\nu_\psi)
& = (\nabla_{\Gamma_\psi}\psi)X\cdot\nu_\psi + \psi \nabla_{\Gamma_\psi}(X\cdot\nu_\psi) \\
& = (\nabla_{\Gamma_\psi}\psi)(1+R_1) + \psi\nabla_{\Gamma_\psi}(1+(X-X\circ\Phi_1^{-1})\cdot\nu\circ\Phi_1^{-1} + X\cdot\tilde{\nu}) \\
& = : (\nabla_{\Gamma_\psi}\psi)(1+R_1) + \psi R_2.
\end{align*}
Recalling \eqref{conti1} and \eqref{conti2},
the $L^\infty$-norm of $R_1$ and $R_2$
can be made as small as we want by taking $\sigma$ small enough,
and in turn from the previous identities we obtain (a).

To prove (b), we first observe that, by reducing $\delta(\sigma)$ if necessary, we can guarantee that
for every $\psi\in\mathcal{D}_{\delta(\sigma)}$
\begin{equation} \label{conti3}
\eta_{\psi} = \eta \circ \Phi_1^{-1} + \tilde{\eta}
\quad\text{with}\quad
|\tilde{\eta}|\leq\sigma
\quad\text{on }\Gamma_\psi\cap\partial\Omega.
\end{equation}
We deduce that on $\Gamma_\psi\cap\partial\Omega$
\begin{align*}
| \nfield\cdot\eta_{\psi} |
= | \psi X\cdot\eta_\psi |
= \big| \psi \bigl( (X\cdot\eta)\circ\Phi_1^{-1} + (X-X\circ\Phi_1^{-1})\cdot\eta\circ\Phi_1^{-1} + X\cdot\tilde{\eta} \bigr) \big|
\leq \e\,|\psi|
\end{align*}
where the last inequality follows by observing that $(X\cdot\eta)\circ\Phi_1^{-1}=0$,
and by \eqref{conti2} and \eqref{conti3} (choosing $\sigma$ small enough, depending on $\e$).
This proves (b).

Finally, by a change of variables (using the area formula \eqref{farea}) we have
\begin{align*}
\|\psi\|^2_{H^1(\Gamma_\psi\cap\Omega)}
= \int_{\Gamma\cap\Omega} \Bigl( |\psi\circ\Phi^\psi|^2 + \frac{|\nabla_\Gamma(\psi\circ\Phi^\psi)|^2}{|D\Phi^\psi[\tau]|^2} \Bigr) J_{\Phi^\psi}\,d\hu\,,
\end{align*}
and (c) follows by \eqref{eq:stimapsi} and recalling that $\psi\circ\Phi^\psi=\psi$ on $\Gamma$.
\end{proof}

Given $\psi\in\mathcal{D}_{\delta_0}$,
we can define a bilinear form on $H^1(\Gamma_\psi\cap\Omega)$ as in \eqref{prscal}, by setting
$$
(\vphi,\vt)_{\sim,\psi} :=
\int_{\Omega\cap\Gamma_\psi} \nabla_{\Gamma_\psi}\vphi\cdot\nabla_{\Gamma_\psi}\vt \,d\hu
 + \int_{\Omega\cap\Gamma_\psi} H_\psi^2\,\vphi\,\vt \,d\hu
 - \int_{\Gamma_\psi\cap\partial\Omega}D\nu_{\partial\Omega}[\nu_\psi,\nu_\psi]\,\vphi\,\vt\,d\huu.
$$
The positivity assumption \eqref{hppos} guarantees that,
if $\delta$ is sufficiently small,
it is possible to control the $H^1$-norm on $\Gamma_\psi$
in terms of the norm $\|\cdot\|_{\sim,\psi}$ associated with $(\cdot,\cdot)_{\sim,\psi}$,
uniformly with respect to $\psi\in\mathcal{D}_\delta$.
This is the content of the following proposition,
whose proof is analogous to \cite[Lemma~5.3]{CMM}
(the only difference lies in the presence of a boundary term, which can be treated similarly to the others).

\begin{proposition} \label{prop:ps2}
In the hypotheses of Theorem~\ref{teo:minW},
there exist $C_1>0$ and $\delta_1\in(0,\delta_0)$ such that for every $\psi\in \mathcal{D}_{\delta_1}$
$$
\|\vphi\|_{H^1(\Gamma_\psi\cap\Omega)} \leq C_1\|\vphi\|_{\sim,\psi} \qquad \text{for every }\vphi\in H^1(\Gamma_\psi\cap\Omega).
$$
\end{proposition}

The previous result allows us to introduce, for $\psi\in\mathcal{D}_{\delta_1}$,
a compact operator $T_\psi:H^1(\Gamma_\psi\cap\Omega)\to H^1(\Gamma_\psi\cap\Omega)$ defined by
\begin{equation} \label{Tf}
(T_\psi\vphi,\vt)_{\sim,\psi} =
-2 \int_{\Gamma_\psi\cap\Omega} \bigl[ v_{\vphi,\psi}^+\div_{\Gamma_\psi}(\vt\nabla_{\Gamma_\psi}u_\psi^+)
- v_{\vphi,\psi}^-\div_{\Gamma_\psi}(\vt\nabla_{\Gamma_\psi}u_\psi^-) \bigr] \,d\hu
\end{equation}
for every $\vphi,\vt\in H^1(\Gamma_\psi\cap\Omega)$,
where $v_{\vphi,\psi}\in H^1_U(\Omega\setminus\Gamma_\psi)$ is the solution to
\begin{equation*}
\int_{\Omega} \nabla v_{\vphi,\psi} \cdot \nabla z\,dx
+\int_{\Gamma_\psi\cap\Omega} \bigl[ z^+\div_{\Gamma_\psi}(\vphi\nabla_{\Gamma_\psi}u_\psi^+) - z^-\div_{\Gamma_\psi}(\vphi\nabla_{\Gamma_\psi}u_\psi^-) \bigr] \,d\hu =0
\end{equation*}
for every $z\in H^1_U(\Omega\setminus\Gamma_\psi)$.
We define also $\lambda_{1,\psi}$ similarly to \eqref{lamda1}.
The following semicontinuity property of the eigenvalues $\lambda_{1,\psi}$
will be crucial in the proof of Theorem~\ref{teo:minW}.
We omit the proof of this result, since it is the same as in \cite[Lemma~5.4]{CMM}:
we only observe that Remark~\ref{rm:Neumannreg} guarantees
that we have the same convergence as in \cite[formula (5.14)]{CMM},
so that we can reproduce word by word the same argument.

\begin{proposition} \label{prop:autovalori}
In the hypotheses of Theorem~\ref{teo:minW},
$$
\limsup_{\|\psi\|_{C^{2}(\Gamma)}\to0} \lambda_{1,\psi}\leq\lambda_1.
$$
\end{proposition}

We are finally ready to prove the local minimality result
stated at the beginning of this section.

\begin{proof}[Proof of Theorem~\ref{teo:minW}]
We divide the proof into two steps.

\smallskip
\noindent
{\it Step 1.}
We first show that
there exist $\delta\in(0,\delta_1)$ and $c>0$ such that
for every $\psi\in\mathcal{D}_\delta$
\begin{equation} \label{claimminW0}
F(\Gamma_\psi,u_\psi)\geq F(\Gamma,u) + c\,\|\psi\|^2_{H^1(\Gamma\cap\Omega)}.
\end{equation}
Given $\psi\in\mathcal{D}_\delta$, with $\delta\in(0,\delta_1)$ to be chosen,
consider the admissible flow $(\Phi_t)_t$ associated with $\psi$,
according to \eqref{eq:flusso},
and its tangent vector field $\nfield$.
Setting $g_\psi(t):=F(\Gamma_{\Phi_t},u_{\Phi_t})$,
we claim that there exist $c>0$ and $\delta>0$ such that
\begin{equation} \label{claimminW}
g_\psi ''(t) \geq 2c \, \|\psi\|^2_{H^1(\Gamma\cap\Omega)}
\qquad\text{for every } t\in[0,1] \text{ and } \psi\in\mathcal{D}_\delta.
\end{equation}
Once this is proved, claim \eqref{claimminW0}
will follow immediately: indeed, as $g_\psi '(0)=0$ since $(\Gamma,u)$ is a critical pair,
and recalling that $\Gamma_{\Phi_1}=\Gamma_\psi$, we deduce
\begin{align*}
F(\Gamma,u) &= g_\psi(0) = g_\psi(1) - \int_0^1(1-t)g_\psi ''(t)\,dt \\
&\leq F(\Gamma_{\psi},u_\psi) - c \, \|\psi\|^2_{H^1(\Gamma\cap\Omega)},
\end{align*}
which is \eqref{claimminW0}.

We now come to the proof of \eqref{claimminW}.
In order to simplify the notation,
we set $\nu_t:=\nu_{\Phi_t}$, $\eta_t:=\eta_{\Phi_t}$, $\Gamma_t:=\Gamma_{\Phi_t}$, and $H_t:=H_{\Phi_t}$.
By Remark~\ref{rm:varIIt}, recalling the definition of $T_{t\psi}$ (see \eqref{Tf}),
we deduce that
\begin{align*}
g_\psi ''(t) &=
 - (T_{t\psi}(\nfield\cdot\nu_t),\nfield\cdot\nu_t)_{\sim,t\psi}
 + \int_{\Gamma_t\cap\Omega} \Bigl( H^2_{t}(\nfield\cdot\nu_t)^2 + |\nabla_{\Gamma_t}(\nfield\cdot\nu_t)|^2 \Bigr)\,d\hu \\
&+\int_{\Gamma_t\cap\Omega} f_t
\bigl( Z_\psi\cdot\nu_t - 2\nfield^\|\cdot\nabla_{\Gamma_t}(\nfield\cdot\nu_t) + D\nu_{t}[\nfield^\|,\nfield^\|] - H_{t}(\nfield\cdot\nu_t)^2 \bigr) \,d\hu \\
&+\int_{\nbordot} (f_t-H_t)(\nfield\cdot\nu_t)(\nfield\cdot\eta_t) \,d\huu
 +\int_{\nbordot} {Z}_{\psi}\cdot\eta_t \,d\huu,
\end{align*}
where $f_t=|\nabla_{\Gamma_t}u_{\Phi_t}^-|^2-|\nabla_{\Gamma_t}u_{\Phi_t}^+|^2+H_{t}$.
Since
\begin{align*}
0 &= {Z}_{\psi}\cdot\norm + D\norm[\nfield,\nfield] \\
  &= {Z}_{\psi}\cdot\norm + D\norm[\nu_t,\nu_t](\nfield\cdot\nu_t)^2
   + \bigl((\nfield\cdot\eta_t)^2\eta_t + 2(\nfield\cdot\nu_t)(\nfield\cdot\eta_t)\nu_t\bigr) \cdot D\norm[\eta_t]
\end{align*}
on $\nbordot$ by Lemma~\ref{lemma:uguaglianze2} (d),
we can rewrite $g_\psi ''(t)$ as
\begin{align}
g_\psi ''(t) &=
 -(T_{t\psi}(\nfield\cdot\nu_t),\nfield\cdot\nu_t)_{\sim,t\psi}
 + \|\nfield\cdot\nu_t\|^2_{\sim,t\psi} \nonumber\\
& + \int_{\Gamma_t\cap\Omega} f_t \bigl( {Z}_\psi\cdot\nu_t - 2\nfield^\|\cdot\nabla_{\Gamma_t}(\nfield\cdot\nu_t) + D\nu_t[\nfield^\|,\nfield^\|] - H_{t}(\nfield\cdot\nu_t)^2 \bigr)\,d\hu \nonumber\\
& + \int_{\nbordot} (f_t-H_t)(\nfield\cdot\nu_t)(\nfield\cdot\eta_t) \,d\huu
 + \int_{\nbordot} {Z}_{\psi}\cdot(\eta_t-\norm) \,d\huu \nonumber\\
& - \int_{\nbordot} \bigl((\nfield\cdot\eta_t)^2\eta_t + 2(\nfield\cdot\nu_t)(\nfield\cdot\eta_t)\nu_t\bigr) \cdot D\norm[\eta_t] \,d\huu\,. \label{varIIg}
\end{align}
We now carefully estimate each term in the previous expression.
In the following, $C$ will denote a generic positive constant, independent of $\psi\in\mathcal{D}_{\delta_1}$,
which may change from line to line.

As $(\Gamma,u)$ satisfies condition \eqref{hppos},
Proposition~\ref{prop:T} implies that $\lambda_1<1$,
so that by Proposition~\ref{prop:autovalori} we can find $\delta_2\in(0,\delta_1)$
such that for every $\psi\in\mathcal{D}_{\delta_2}$
\begin{equation} \label{stima-2}
\lambda_{1,\psi} < \frac12 (\lambda_1+1) < 1.
\end{equation}
Fix $\e>0$ to be chosen later, and let $\delta(\e)>0$ be given by Lemma~\ref{lemma:stimecampo}
(assume without loss of generality that $\delta(\e)<\delta_2$).
We remark that, if $\psi\in\mathcal{D}_{\delta(\e)}$, then $t\psi\in\mathcal{D}_{\delta(\e)}$ for every $t\in[0,1]$,
and $X_{t\psi}=tX_\psi$: hence we can apply (a) and (b) of Lemma~\ref{lemma:stimecampo} to $t\psi$, and we easily obtain that
\begin{equation} \label{stima0}
\frac12\|\psi\|^2_{H^1(\Gamma\cap\Omega)}
\leq \|\nfield\cdot\nu_t\|^2_{H^1(\Gamma_t\cap\Omega)}
\leq 2\|\psi\|^2_{H^1(\Gamma\cap\Omega)},
\end{equation}
\begin{equation} \label{stima0bis}
|\nfield\cdot\eta_t| \leq \e\, |\psi|
\qquad\text{on }\Gamma_t\cap\partial\Omega,
\end{equation}
for every $\psi\in\mathcal{D}_{\delta(\e)}$ and for every $t\in[0,1]$.

Fix now any $\psi\in\mathcal{D}_{\delta(\e)}$.
From the definition of $\lambda_{1,\psi}$ and \eqref{stima-2} we have
\begin{align} \label{stima1}
-(T_{t\psi} &(\nfield\cdot\nu_t),\nfield\cdot\nu_t)_{\sim,t\psi} + \|\nfield\cdot\nu_t\|^2_{\sim,t\psi}
\geq (1-\lambda_{1,t\psi}) \|\nfield\cdot\nu_t\|^2_{\sim,t\psi} \nonumber\\
&\geq \frac{1-\lambda_1}{2} \|\nfield\cdot\nu_t\|^2_{\sim,t\psi}
\geq \frac{1-\lambda_1}{2C_1^2} \|\nfield\cdot\nu_t\|^2_{H^1(\Gamma_t\cap\Omega)}
\geq \frac{1-\lambda_1}{4C_1^2} \|\psi\|^2_{H^1(\Gamma\cap\Omega)},
\end{align}
where the last two inequalities follow from Proposition~\ref{prop:ps2} and from \eqref{stima0}.

By Remark~\ref{rm:Neumannreg} the map
$$
\psi\in\mathcal{D}_{\delta(\e)}\mapsto \| |\nabla_{\Gamma_\psi}u_{\psi}^-|^2-|\nabla_{\Gamma_\psi}u_{\psi}^+|^2+H_{\psi} \|_{L^\infty(\Gamma_\psi\cap\Omega)}
$$
is continuous with respect to the $C^2$-topology;
hence, as it vanishes for $\psi=0$ by \eqref{EL1},
possibly reducing $\delta(\e)$ we have that for every $\psi\in\mathcal{D}_{\delta(\e)}$
$$
\| |\nabla_{\Gamma_\psi}u_{\psi}^-|^2-|\nabla_{\Gamma_\psi}u_{\psi}^+|^2+H_{\psi} \|_{L^\infty(\Gamma_\psi\cap\Omega)} <\e.
$$
We deduce that
\begin{align} \label{stima2}
\int_{\Gamma_t\cap\Omega} &f_t
    \bigl( {Z}_\psi\cdot\nu_t - 2\nfield^\|\cdot\nabla_{\Gamma_t}(\nfield\cdot\nu_t) + D\nu_t[\nfield^\|,\nfield^\|] - H_{t}(\nfield\cdot\nu_t)^2 \bigr) \,d\hu \nonumber\\
& \geq - \e \,
    \| {Z}_\psi\cdot\nu_t - 2\nfield^\|\cdot\nabla_{\Gamma_t}(\nfield\cdot\nu_t) + D\nu_t[\nfield^\|,\nfield^\|] + H_{t}(\nfield\cdot\nu_t)^2 \|_{L^{1}(\Gamma_t\cap\Omega)} \nonumber \\
& \geq - \e \, C \,
    \Bigl(
    \|\psi\|_{L^{2}(\Gamma_t\cap\Omega)}^2 +
    \|\nabla_{\Gamma_t}(\nfield\cdot\nu_t)\|_{L^2(\Gamma_t\cap\Omega)} \, \|\psi\|_{L^{2}(\Gamma_t\cap\Omega)} \Bigr) \nonumber\\
& \geq - \e \, C \, \| \psi \|_{H^1(\Gamma\cap\Omega)}^2,
\end{align}
where we used also \eqref{stima0ter}, \eqref{stima0}, and (c) of Lemma~\ref{lemma:stimecampo}.

By \eqref{stima0ter}, \eqref{stima0bis} and (c) of Lemma~\ref{lemma:stimecampo} we have
\begin{align} \label{stima3}
\int_{\nbordot} (f_t-H_t)(\nfield\cdot\nu_t)(\nfield\cdot\eta_t) \,d\huu
\geq - \e \, C \, \int_{\Gamma_t\cap\partial\Omega}\psi^2\,d\huu
\geq - \e \, C \, \|\psi\|^2_{H^1(\Gamma\cap\Omega)} \,.
\end{align}

By reducing $\delta(\e)$ if necessary we can assume
\begin{equation*}
\max_{x\in\Gamma_\psi\cap\partial\Omega} |\eta_\psi(x)-\norm(x)|<\e \qquad\text{for every }\psi\in\mathcal{D}_{\delta(\e)},
\end{equation*}
so that using again \eqref{stima0ter} and (c) of Lemma~\ref{lemma:stimecampo} we obtain
\begin{align} \label{stima4}
\int_{\nbordot} {Z}_{\psi}\cdot(\eta_t-\norm)\,d\huu
\geq - \e \, C \, \int_{\Gamma_t\cap\partial\Omega}\psi^2\,d\huu
\geq - \e \, C \, \|\psi\|^2_{H^1(\Gamma\cap\Omega)} \,.
\end{align}

Finally, we proceed in a similar way to estimate the last integral in \eqref{varIIg}:
by \eqref{stima0ter} and \eqref{stima0bis}
\begin{align}\label{stima6}
- \int_{\nbordot} &\bigl( (\nfield\cdot\eta_t)^2\eta_t + 2(\nfield\cdot\nu_t)(\nfield\cdot\eta_t)\nu_t \bigr) \cdot D\norm[\eta_t] \,d\huu
\geq - \e \, C \, \|\psi\|^2_{H^1(\Gamma\cap\Omega)}.
\end{align}

Collecting \eqref{stima1}--\eqref{stima6},
by \eqref{varIIg} we conclude that for every $\psi\in\mathcal{D}_{\delta(\e)}$ and for every $t\in[0,1]$
$$
g_\psi ''(t) \geq \biggl( \frac{1-\lambda_1}{4C_1^2} - \e\,C \biggr) \|\psi\|^2_{H^1(\Gamma\cap\Omega)}
$$
for some positive constant $C$ (independent of $\psi$),
so that by choosing $\e$ sufficiently small we obtain the claim \eqref{claimminW}
and, in turn, \eqref{claimminW0}.

\smallskip
\noindent
{\it Step 2.}
The conclusion of the theorem follows now by approximation:
given any $\psi\in W^{2,\infty}(\Gamma\cap\Omega)$ with $\|\psi\|_{W^{2,\infty}(\Gamma\cap\Omega)}<\delta$,
we can find a sequence $\psi_n\in\mathcal{D}_\delta$ converging to $\psi$ in $W^{1,\infty}(\Gamma\cap\Omega)$
for which the conclusion obtained in the previous step holds:
$$
F(\Gamma_{\psi_n},u_{\psi_n}) \geq F(\Gamma,u) + c\,\|\psi_n\|^2_{H^1(\Gamma\cap\Omega)}.
$$
By passing to the limit in the previous inequality,
and noting that $F(\Gamma_{\psi_n},u_{\psi_n})\to F(\Gamma_{\psi},u_{\psi})$
as a consequence of the $W^{1,\infty}$-convergence of $\psi_n$ to $\psi$,
we conclude that the same estimate holds for $\psi$.
Hence the conclusion of the theorem follows by recalling that $F(\Gamma_\psi,v)\geq F(\Gamma_\psi,u_\psi)$
for every $v\in H^1(\Omega\setminus\Gamma_\psi)$ with $v=u$ in $(\Omega\setminus U)\cup\partial_D\Omega$.
\end{proof}

\end{section}


\begin{section}{Local $C^{1,\alpha}$-minimality} \label{sect:minW1inf}

In this section we show that the $W^{2,\infty}$-minimality property proved in the previous section
implies that $(\Gamma,u)$ is also a minimizer with respect to small $C^{1,\alpha}$-perturbations of the discontinuity set.
We start by a preliminary construction that will be needed in the proof.

\begin{remark} \label{rm:intorni}
Let $X$ be the vector field defined at the beginning of Section~\ref{sect:minW},
which, we recall, coincides with $\nu$ on $\Gamma$ and is tangent to $\partial\Omega$,
and let $\Psi$ be the flow generated by $X$.
We want to define a one-parameter family of smooth curves $(\Gamma_\delta)_\delta$, for $\delta\in(-\delta_0,\delta_0)$, with $\Gamma_0=\Gamma$, such that $X$ is normal to each curve of the family,
and whose union is a tubular neighborhood of $\Gamma$.
In order to do this, let $x_0\in\Gamma\cap\partial\Omega$ and let $x_\delta:=\Psi(\delta,x_0)$.
We then define $\Gamma_\delta$ as the trajectory of the flow generated by $X^\bot$ starting from $x_\delta$,
where the vector field $X^\bot$ is obtained by a rotation of $X$ by $\frac{\pi}{2}$.
This construction provides a family of curves with the desired properties.

We can then define a family of tubular neighborhoods of $\Gamma$ in $\Omega$
whose boundaries meet $\partial\Omega$ orthogonally, by setting for $\delta\in(-\delta_0,\delta_0)$
$$
\mathcal{I}_\delta(\Gamma) := \bigcup_{|s|<\delta}\Gamma_s\,.
$$
\end{remark}

\begin{proposition}\label{prop:minw1}
Let $(\Gamma,u)$ be a strictly stable regular critical pair in an admissible subdomain $U$,
and let $\alpha\in(0,1)$.
There exists $\delta>0$ such that
$$
F(\Gamma,u) < F(\Phi(\Gamma),v)
$$
for every diffeomorphism $\Phi\in C^{1,\alpha}(\ombar;\ombar)$
with $0<\|\Phi-Id\|_{C^{1,\alpha}(\ombar)}\leq\delta$
and $\supp(\Phi-Id)\subset\subset U\setminus\partial_D\Omega$,
and for every $v\in H^1(\Omega\setminus\Phi(\Gamma))$
with $v=u$ in $(\Omega\setminus U)\cup\partial_D\Omega$.
\end{proposition}

\begin{proof}
Assume by contradiction that there exist sequences $\sigma_n\to0$ and $\Phi_n\in C^{1,\alpha}(\ombar;\ombar)$, with
$$
\supp(\Phi_n-Id)\subset\subset U\setminus\partial_D\Omega, \quad 0<\|\Phi_n-Id\|_{C^{1,\alpha}(\ombar)}\leq\sigma_n,
$$
such that $F(\Phi_n(\Gamma),u_n) \leq F(\Gamma,u)$,
where $u_n:=u_{\Phi_n}$ is defined as in \eqref{uphi}.
Notice that, arguing as in Remark~\ref{rm:Neumannreg},
we have that $u_n^\pm$ are of class $C^{1,\alpha}$ up to $\Phi_n(\Gamma)$, and
$$
\| \nabla_\Gamma(u_n^\pm\circ\Phi_n)-\nabla_\Gamma u^\pm \|_{L^\infty(\Omega^\pm)}\to0.
$$

We first extend $u^+$ and $u^-$ to $C^{1,\alpha}$-functions in $\Omega^-$ and $\Omega^+$, respectively,
by using \cite[Theorem~6.2.5]{Gri}.
We similarly extend $u_n^\pm\circ\Phi_n$ to $C^{1,\alpha}$-functions $\tilde{u}_n^\pm$ in $\Omega^\mp$,
and we set $v_n^\pm := \tilde{u}_n^\pm\circ\Phi_n^{-1}$:
since the extension operator constructed in \cite[Theorem~6.2.5]{Gri}
is continuous, we have that
$$
\|\nabla v_n^\pm - \nabla u^\pm\|_{L^\infty(\Omega)}\leq\delta_n
$$
for some $\delta_n\to0$.
Finally, as $\|\Phi_n-Id\|_{C^{1,\alpha}}\to0$, we can also assume that $\Phi_n(\Gamma)\subset\mathcal{I}_{\delta_n}(\Gamma)$.

Consider the following obstacle problems
\begin{align} \label{ostacolo}
\min \Bigl\{ J(E,v^+,v^-) : E&\subset\Omega, \;
    \Omega^+\setminus\mathcal{I}_{\delta_n}(\Gamma) \subset E \subset \Omega^+\cup\mathcal{I}_{\delta_n}(\Gamma), \;
    v^\pm-u^\pm\in W^{1,\infty}(\Omega), \nonumber\\
    &v^+\chi_E+v^-\chi_{E^c}=u \,\text{ in }\, (\Omega\setminus U)\cup\partial_D\Omega, \;
    \|\nabla v^\pm-\nabla u^\pm\|_{L^\infty(\Omega)}\leq 1 \Bigr\}\,,
\end{align}
where
$$
J(E,v^+,v^-):= \int_E |\nabla v^+|^2 + \int_{\Omega\setminus E} |\nabla v^-|^2 + P(E,\Omega),
$$
and let $(F_n,w_n^+,w_n^-)$ be a solution to \eqref{ostacolo} (which exists by the direct method of the Calculus of Variations).
Since $(\Phi_n(\Omega^+),v_n^+,v_n^-)$ is an admissible competitor, we deduce that
\begin{equation} \label{contracalpha}
J(F_n,w_n^+,w_n^-)\leq J(\Phi_n(\Omega^+),v_n^+,v_n^-)=F(\Phi_n(\Gamma),u_n)\leq F(\Gamma,u).
\end{equation}
We now divide the proof into three steps.

\smallskip
\noindent
{\it Step 1.}
We claim that, if $\gamma>0$ is sufficiently large (independently of $n$), then $(F_n,w_n^+,w_n^-)$ is also a solution to
\begin{align} \label{ostacolo2}
\min \Bigl\{ \widetilde{J}(E,v^+,v^-) \, : \;
    E\subset\Omega, \;
    v^\pm-u^\pm \in W^{1,\infty}(\Omega), \;
    &\|\nabla v^\pm-\nabla u^\pm\|_{L^\infty(\Omega)}\leq 1, \nonumber\\
    &v^\pm=u \,\text{ in }\, (\Omega^\pm\setminus U)\cup(\partial_D\Omega\cap\overline{\Omega}^\pm) \Bigr\},
\end{align}
where
$$
\widetilde{J}(E,v^+,v^-):= \int_E |\nabla v^+|^2 + \int_{\Omega\setminus E} |\nabla v^-|^2 + P(E,\Omega) + \gamma|E\bigtriangleup T_n(E)|
$$
and $T_n(E):= E\cup(\Omega^+\setminus\mathcal{I}_{\delta_n}(\Gamma))\cap(\Omega^+\cup\mathcal{I}_{\delta_n}(\Gamma))$.

In order to prove this, fix any competitor $(F,w^+,w^-)$ for problem \eqref{ostacolo2}.
We denote by $\nu_E$ the generalized inner normal to a finite perimeter set $E$.
Since $\nu_{T_n(F)}=X$ almost everywhere on $\partial^*T_n(F)\cap \Gamma_{\delta_n}$, and $|X|\leq1$,
we can estimate the difference of the perimeters of $F$ and $T_n(F)$ in $\Omega^+$ as follows:
\begin{align*}
P(F,\Omega^+)&-P(T_n(F),\Omega^+)
= \int_{(\partial^*F\setminus\partial^*T_n(F))\cap\Omega^+}\,d\hu
   - \int_{(\partial^*T_n(F)\setminus\partial^*F)\cap\Omega^+}\,d\hu \\
&\geq \int_{(\partial^*F\setminus\partial^*T_n(F))\cap\Omega^+} X\cdot\nu_F\,d\hu
   - \int_{(\partial^*T_n(F)\setminus\partial^*F)\cap\Omega^+} X\cdot\nu_{T_n(F)}\,d\hu \\
&= \int_{(F\bigtriangleup T_n(F))\cap\Omega^+}\div X
 \geq - \|\div X\|_\infty|(F\bigtriangleup T_n(F))\cap\Omega^+|,
\end{align*}
where we used the divergence theorem taking into account that $X\cdot\norm=0$ on $\partial\Omega$.
A similar estimate holds in $\Omega^-$, and we conclude that
$$
P(F,\Omega)-P(T_n(F),\Omega) \geq - \|\div X\|_\infty|F\bigtriangleup T_n(F)|.
$$
Since $\nabla w^\pm$ are uniformly bounded in $L^\infty$
by a constant $\Lambda$ depending only on $\|\nabla u\|_\infty$,
we have for the volume terms
$$
\bigg|\int_{F}|\nabla w^+|^2 - \int_{T_n(F)}|\nabla w^+|^2\bigg| \leq \Lambda^2|F\bigtriangleup T_n(F)|,
$$
and a similar estimate holds for $w^-$ in the complements of the sets $F$ and $T_n(F)$.
Hence we deduce by minimality of $(F_n,w_n^+,w_n^-)$
\begin{align*}
\widetilde{J}(F,w^+,w^-) - \widetilde{J}(F_n,w^+_n,w_n^-)
&\geq J(F,w^+,w^-)-J(T_n(F),w^+,w^-)+\gamma|F\bigtriangleup T_n(F)| \\
&\geq \bigl( \gamma-2\Lambda^2-\|\div X\|_\infty \bigr)|F\bigtriangleup T_n(F)|\geq0
\end{align*}
if $\gamma>2\Lambda^2+\|\div X\|_\infty$.
This shows that $(F_n,w_n^+,w_n^-)$ is also a solution to \eqref{ostacolo2}.

\smallskip
\noindent
{\it Step 2.}
Each set $F_n$ is a quasi-minimizer of the area functional in $\Omega$, that is there exists a constant $\omega>0$ (independent of $n$) such that for every set $F\subset\Omega$ of finite perimeter with $F\bigtriangleup F_n\subset\subset B_r(x)$ ($x\in\R^2$, $r>0$) we have
\begin{equation} \label{perquamin}
P(F_n,\Omega) \leq P(F,\Omega) + \omega |F\bigtriangleup F_n|.
\end{equation}
This can be directly deduced from the fact that $(F_n,w_n^+,w_n^-)$ solves \eqref{ostacolo2},
using in particular the $L^\infty$ bound on $\nabla w_n^\pm$ to estimate the Dirichlet integrals by $|F\bigtriangleup F_n|$.

Combining this information with the Hausdorff convergence of $F_n$ to $\Omega^+$
(whose boundary inside $\Omega$ is regular),
we deduce that each $F_n$ has $C^{1,\frac12}$ boundary inside $\Omega$ (for $n$ sufficiently large)
which converges to $\Gamma$ in the $C^{1,\alpha}$- sense for all $\alpha\in(0,\frac12)$.
This is a well-known consequence of the classical regularity theory of quasi-minimizers of the area functional,
which is stated in our precise setting in \cite[Theorem~3.5]{JP} (see also the references contained therein).
In particular, the regularity up to the boundary $\partial\Omega$ follows from a work by Gr\"{u}ter \cite{Gru},
which guarantees in addition that the intersection of $\overline{\partial F_n\cap\Omega}$ with $\partial\Omega$ is orthogonal.

Hence there exist diffeomorphisms $\Psi_n:\ombar\to\ombar$ of class $C^{1,\alpha}$
such that $F_n=\Psi_n(\Omega^+)$, $\overline{\partial F_n\cap\Omega}=\Psi_n(\Gamma)$ and $\|\Psi_n-Id\|_{C^{1,\alpha}(\Gamma)}\to0$.
In turn, by Lemma~\ref{lemma:grafici} we conclude that $\overline{\partial F_n\cap\Omega}=\Gamma_{\psi_n}$
for some functions $\psi_n\in C^{1,\alpha}(\Gamma)$ such that $\psi_n\to0$ in $C^{1,\alpha}(\Gamma)$.

We also remark that $\nabla w_n^\pm$ are H\"{o}lder continuous up to $\Gamma_{\psi_n}$,
and they converge uniformly to $\nabla u^\pm$.
Indeed, by considering the Dirichlet minimizer $u_{\Psi_n}$ in $\Omega\setminus\Psi_n(\Gamma)$
under the usual boundary conditions, we have by elliptic regularity
(as in Remark~\ref{rm:Neumannreg}) that
$\nabla_\Gamma(u_{\Psi_n}^\pm\circ\Psi_n)$ is H\"{o}lder continuous and converges uniformly to $\nabla_\Gamma u^\pm$.
Hence, for $n$ large enough, and also taking into account the continuity of the extension operator,
$u_{\Psi_n}$ satisfies the constraint $\|\nabla u_{\Psi_n}^\pm- \nabla u^\pm\|_\infty \leq 1$
so that we conclude that $w_n^\pm=u_{\Psi_n}^\pm$.

\smallskip
\noindent
{\it Step 3.}
By the quasi-minimality property \eqref{perquamin} and the $C^{1,\alpha}$-regularity of $\Gamma_{\psi_n}$,
we deduce by a standard argument (see, \emph{e.g.}, \cite[Proposition~3.2]{JP})
that the curvatures $H_{\psi_n}$ of the sets $\Gamma_{\psi_n}$ are uniformly bounded by the constant $\omega$.
In turn, this provides the $W^{2,\infty}$-regularity of $\Gamma_{\psi_n}$.

If we now write the Euler-Lagrange equations for problem \eqref{ostacolo}, we get
$$
H_{\psi_n} =
\left\{
  \begin{array}{ll}
    |\nabla w_n^+|^2 - |\nabla w_n^-|^2 & \hbox{on }\Gamma_{\psi_n}\cap\mathcal{I}_{\delta_n}(\Gamma), \\
    H_{\Gamma_{\pm\delta_n}} & \hbox{on }\Gamma_{\psi_n}\cap\Gamma_{\pm\delta_n},
  \end{array}
\right.
$$
where $H_{\Gamma_{\pm\delta_n}}$ denotes the curvature of the curve $\Gamma_{\pm\delta_n}$.
Moreover, as $(\Gamma,u)$ is a critical pair, by \eqref{EL1} we have
\begin{equation*}
H_\Gamma
= |\nabla u^+|^2 - |\nabla u^-|^2
\quad\text{on }\Gamma.
\end{equation*}
Hence, by the uniform convergence of $\nabla w_n^\pm$ to $\nabla u^\pm$
and observing that the curvature $H_{\Gamma_{\pm\delta_n}}$ is uniformly close to $H_\Gamma$
by the regularity of the flow generating the family of curves $(\Gamma_\delta)_\delta$,
we deduce that
$$
\|H_{\psi_n}\circ\Psi_n - H_\Gamma\|_{L^\infty(\Gamma)} \to0
\qquad\text{as }n\to\infty,
$$
which implies, by Lemma~\ref{lemma:grafici}, that $\|\psi_n\|_{W^{2,\infty}(\Gamma)}\to0$.

We can conclude that, setting $w_n:=w_n^+\chi_{F_n}+w_n^-\chi_{F_n^c}$, by \eqref{contracalpha}
$$
F(\Gamma_{\psi_n},w_n) = J(F_n,w_n^+,w_n^-) \leq F(\Gamma,u),
$$
which implies, by the isolated local $W^{2,\infty}$-minimality of $(\Gamma,u)$ proved in Theorem~\ref{teo:minW},
that for all $n$ large enough $\psi_n=0$ and $w_n=u$.
As a consequence, $(\Phi_n(\Omega^+),v_n^+,v_n^-)$ is itself a solution to \eqref{ostacolo}:
by repeating all the previous argument for this sequence instead of $(F_n,w_n^+,w_n^-)$,
we finally reach a contradiction to the isolated local $W^{2,\infty}$-minimality of $(\Gamma,u)$.
\end{proof}

\end{section}


\begin{section}{Proof of Theorem~\ref{teo:minSBV}} \label{sect:minSBV}

In this section we complete the proof of Theorem~\ref{teo:minSBV}.
It will be useful to introduce the relaxed functional
$$
\F(u):= \int_{\Omega} |\nabla u|^2\,dx + \hu(S_u) \qquad\text{for }u\in \sbv(\Omega)
$$
and, for $B\subset\Omega$ Borel set, its local version
\begin{align*}
\F(u;B):= \int_{\Omega\cap B} |\nabla u|^2\,dx + \hu(S_u\cap B).
\end{align*}

\begin{remark} \label{rm:sbv}
We recall here that if $(K,v)\in\spazio$ is an admissible pair with $\mathcal{H}^{N-1}(K)<+\infty$ and $v\in L^\infty(\Omega)$,
then the function $v$ is in $\sbv(\Omega)$ and satisfies $\hu(S_v\setminus K)=0$ (see \cite[Proposition~4.4]{AFP});
in particular, we have $\F(v)\leq F(K,v)$.
On the other hand, if $(\Gamma,u)$ is a regular critical pair, then $u\in\sbv(\Omega)$, $\overline{S}_u=\Gamma$ and $\F(u)=F(\Gamma,u)$.
\end{remark}

\begin{remark} \label{rm:ortogonali}
We observe that, in proving Theorem~\ref{teo:minSBV},
we can assume without loss of generality that $U$ is an open set of class $C^\infty$
and that $\partial U$ and $\partial \Omega$ are orthogonal where they intersect.
Indeed, assume to have proved the theorem under these additional assumptions.
If $U$ is any admissible subdomain for $(\Gamma,u)$,
we can find a decreasing sequence of admissible subdomains $U_n$ of class $C^\infty$,
with boundaries meeting $\partial\Omega$ orthogonally, such that $U$ is the interior part of $\bigcap_nU_n$.
It follows from Corollary~\ref{cor:largerdomain} that the second variation is strictly positive in $U_n$ for $n$ large enough,
and hence $(\Gamma,u)$ is an isolated local minimizer in $U_n$.
This immediately yields the conclusion also in the initial domain $U$.
\end{remark}

We can now start the proof of Theorem~\ref{teo:minSBV}.
By Remark~\ref{rm:ortogonali} we are allowed to perform the proof under the additional assumption
that $U$ has boundary of class $C^\infty$ intersecting $\partial\Omega$ orthogonally.
Moreover, from Remark~\ref{rm:sbv} it follows that in order obtain the result
it is sufficient to show that there exists $\delta>0$ such that
$\F(u)<\F(v)$ for every $v\in\sbv(\Omega)$ with $v=u$ in $(\Omega\setminus U)\cup\partial_D\Omega$ and $0<\|v-u\|_{L^1(\Omega)}<\delta$.

Hence we assume by contradiction that there exists a sequence $v_n\in\sbv(\Omega)$,
with $v_n=u$ in $(\Omega\setminus U)\cup\partial_D\Omega$, such that $0<\|v_n-u\|_{L^1(\Omega)}\to0$ and
\begin{equation} \label{contraminloc}
\F(v_n) \leq \F(u).
\end{equation}
By a truncation argument, we can assume that
$\|v_n\|_{L^\infty(\Omega)}\leq\|u\|_{L^\infty(\Omega)}=:M<+\infty$.

We introduce a bounded open set $\Omega'$ such that $\Omega\subset\Omega'$ and $\Omega'\cap\partial\Omega=\partial_D\Omega$,
in order to enforce the boundary condition on $\partial_D\Omega$.
We can extend $u$ in $\Omega'\setminus\Omega$ to a function $u\in\sbv(\Omega')$ such that $\hu(S_u\cap\partial_D\Omega)=0$
and $\|u\|_{L^\infty(\Omega')}\leq M$. Moreover, we can also assume that $v_n\in\sbv(\Omega')$ and $v_n=u$ in $\Omega'\setminus(U\cap\Omega)$.
In particular, $\hu(S_{v_n}\cap\partial_D\Omega)=0$ and hence $\F(v_n;\Omega')\leq\F(u;\Omega')$.

We set $\e_n:=\|v_n-u\|_{L^2(\Omega)}^2\to0$,
$$
h_n(t):=
\left\{
  \begin{array}{ll}
    \sqrt{(t-\e_n)^2+\e_n^2}-\e_n, & \hbox{if }t>\e_n, \\
    0, & \hbox{if }0\leq t \leq \e_n,
  \end{array}
\right.
$$
and we consider, for $\beta>0$ to be chosen later, a solution $w_n$ to the following penalized minimum problem:
\begin{equation} \label{penalizzaton}
\min \Bigl\{ \F(w;\Omega') + \beta h_n \bigl(\|w-u\|_{L^2(\Omega)}^2 \bigr) \, : \, w\in\sbv(\Omega'), \, w=u \text{ in } \Omega'\setminus(U\cap\Omega) \Bigr\} \,.
\end{equation}
The existence of a solution to \eqref{penalizzaton} is guaranteed by the lower semi-continuity and compactness theorems in $\sbv$
(see \cite[Theorem~4.7 and Theorem~4.8]{AFP}),
and we can also assume $\|w_n\|_{L^\infty(\Omega')}\leq M$.
Observe in addition that, by \eqref{contraminloc} and by minimality of $w_n$, we have
\begin{equation} \label{contraminloc2}
\F(w_n;\Omega') \leq \F(w_n;\Omega') + \beta h_n \bigl( \|w_n-u\|_{L^2(\Omega)}^2 \bigr)
\leq \F(v_n;\Omega') \leq \F(u;\Omega') \,,
\end{equation}
and thus the energies $\F(w_n;\Omega')$ are equibounded.
In turn, again by the compactness and lower semi-continuity theorems in $\sbv$ we deduce that, up to subsequences,
$w_n$ converges in $\sbv(\Omega')$ (see Definition~\ref{def:convsbv}) and in $L^p(\Omega')$ for every $p\in[1,+\infty)$
to a function $z\in\sbv(\Omega')$ which solves the minimum problem
\begin{equation} \label{penalizzato}
\min \Bigl\{ \F(w;\Omega') + \beta \int_{\Omega}|w-u|^2\,dx \, : \, w\in\sbv(\Omega'), \, w=u \text{ in } \Omega'\setminus(U\cap\Omega) \Bigr\} \,.
\end{equation}
Indeed, if $w\in\sbv(\Omega')$ is an admissible function for problem \eqref{penalizzato},
then by semi-continuity and by minimality of $w_n$ we immediately deduce that
\begin{align*}
\F(z;\Omega') + \beta \int_{\Omega}|z-u|^2\,dx
& \leq \liminf_{n\to\infty} \biggl( \F(w_n;\Omega') + \beta h_n \bigl( \|w_n-u\|_{L^2(\Omega)}^2 \bigr) \biggr) \\
& \leq \liminf_{n\to\infty} \biggl( \F(w;\Omega') + \beta h_n \bigl( \|w-u\|_{L^2(\Omega)}^2 \bigr) \biggr)\,.
\end{align*}

By the result in \cite{Mor}, based on the construction of a suitable calibration,
we have the following result.

\begin{proposition} \label{prop:calibrazioni}
If $\beta$ is sufficiently large, then the unique solution to \eqref{penalizzato} is $u$ itself.
\end{proposition}

Notice that in \cite{Mor} only pure Neumann boundary conditions are considered
(\emph{i.e.}, $\partial_D\Omega=\emptyset$).
Nevertheless, exactly the same construction applies to our setting without any change
(see also \cite[Remark~4.3.5]{Mor2}).

Hence, by choosing $\beta>0$ sufficiently large, we have that $w_n\to u$ in $SBV(\Omega')$.
In addition, by lower semi-continuity of $\F$ and by \eqref{contraminloc2} we deduce that
$\F(w_n;\Omega') \to \F(u;\Omega')$ as $n\to\infty$,
which combined with the lower semi-continuity of the two terms in the functional
(which holds separately, by \cite[Theorem~4.7]{AFP}) yields
\begin{equation} \label{convergenze}
\lim_{n\to+\infty} \int_{\Omega'}|\nabla w_n|^2\,dx = \int_{\Omega'} |\nabla u|^2\,dx, \qquad
\lim_{n\to+\infty} \hu(S_{w_n}) = \hu(S_u).
\end{equation}
In the following lemma we localize the previous convergence in open sets
and we prove a continuity property that will be used subsequently.

\begin{lemma} \label{lemma:convergenzeloc}
For every open set $A\subset\R^2$ such that $|\partial A|=0$ and $\hu(S_u\cap\partial A)=0$ we have
$$
\int_{\Omega'\cap A}|\nabla w_n|^2\,dx \to \int_{\Omega'\cap A} |\nabla u|^2\,dx,
\quad
\hu(S_{w_n}\cap A) \to \hu(S_u\cap A)
$$
as $n\to+\infty$.
Moreover, for every bounded continuous function $f\in C^0(\Omega')$ we have
$$
\int_{S_{w_n}\cap A} f\,d\hu \to \int_{S_u\cap A}f\,d\hu.
$$
\end{lemma}

\begin{proof}
The first part of the statement follows easily from the lower semi-continuity of both terms in the functional,
which holds in every open set, combined with \eqref{convergenze}.
To prove the second part,
fix any continuous and bounded function $f:\Omega'\to\R$.
Assuming without loss of generality that $f\geq0$ (for the general case, one can split $f$ into positive and negative parts),
we have to show that
$$
\int_0^{\max f} \hu(S_{w_n} \cap A \cap \{f>t\})\,dt \to \int_0^{\max f} \hu(S_{u} \cap A \cap \{f>t\})\,dt.
$$
The sets $A_t=\{f>t\}$ are open and they satisfy $|\partial A_t|=0$, $\hu(S_u\cap\partial A_t)=0$ for all except at most for countable many $t$. Then, by the assumptions on $A$, the same is true for the sets $A\cap A_t$,
and hence by the first part of the lemma we have
$$
\hu(S_{w_n}\cap A \cap\{f>t\})\to\hu(S_u\cap A\cap\{f>t\}) \qquad\text{for a.e. }t\in(0,\max f)\,,
$$
and by the Dominated Convergence Theorem we obtain the conclusion.
\end{proof}

In the following proposition we show that $w_n$ satisfies a quasi-minimality property.
This is an essential step in our strategy to prove Theorem~\ref{teo:minSBV}:
indeed, as a consequence of the regularity theory for quasi-minimizers of the Mumford-Shah functional
we obtain firstly that a uniform lower bound on the 1-dimensional density of $S_{w_n}$ holds,
and moreover we will be able to deduce the $C^{1,\alpha}$-convergence of $S_{w_n}$ to $S_u$ (see Proposition~\ref{prop:david}).

\begin{proposition} \label{prop:quasimin}
Each function $w_n$ is a quasi-minimizer of the Mumford-Shah functional, that is,
there exists a positive constant $\omega$ (independent of $n$)
such that if $x\in\ombar'$ and $\rho>0$ then
\begin{equation} \label{quasimin}
\F(w_n;B_\rho(x)\cap\Omega') \leq \F(v;B_\rho(x)\cap\Omega') + \omega\rho^2
\end{equation}
for every $v\in\sbv(\Omega')$ with $v=u$ in $\Omega'\setminus(U\cap\Omega)$ and $\{v\neq w_n\}\subset\subset B_\rho(x)$.
\end{proposition}

\begin{proof}
Let $v$ be as in the statement, and set $v^M:=(-M)\vee(v\wedge M)$ (where $M=\|u\|_\infty$).
Then, since $v^M\in\sbv(\Omega')$ is an admissible competitor in problem \eqref{penalizzaton},
$\{v^M\neq w_n\}\subset\{v\neq w_n\}$ (as $\|w_n\|_{\infty}\leq M$) and $\F(v^M)\leq\F(v)$,
we have by minimality of $w_n$
\begin{align*}
\F(w_n;B_\rho(x)\cap\Omega')
&\leq \F(v^M;B_\rho(x)\cap\Omega')
 + \beta h_n \biggl(\int_{\Omega} |v^M-u|^2\,dy\biggr) - \beta h_n \biggl(\int_\Omega|w_n-u|^2\,dy\biggr)\\
&\leq \F(v;B_\rho(x)\cap\Omega')
 + \beta \bigg|\int_{B_\rho(x)\cap\Omega} |v^M-u|^2\,dy - \int_{B_\rho(x)\cap\Omega} |w_n-u|^2\,dy \bigg|\\
&\leq \F(v;B_\rho(x)\cap\Omega') + 8M^2\beta\pi\rho^2 \,,
\end{align*}
where we used the fact that $h_n$ is 1-Lipschitz in the second inequality.
Hence \eqref{quasimin} follows by choosing $\omega := 8M^2\beta\pi$.
\end{proof}

\begin{corollary} \label{cor:kuratowsky}
Each set $S_{w_n}$ is essentially closed: $\hu(\overline{S}_{w_n}\setminus S_{w_n})=0$.
Moreover, the sets $\overline{S}_{w_n}$ converge to $\overline{S}_u$ in $\ombar'$ in the sense of Kuratowski:
\begin{itemize}
  \item [(i)] for every $x_n\in \overline{S}_{w_n}$ such that $x_n\to x$, then $x\in \overline{S}_u$;
  \item [(ii)] for every $x\in \overline{S}_u$ there exist $x_n\in \overline{S}_{w_n}$ such that $x_n\to x$.
\end{itemize}
\end{corollary}

\begin{proof}
Thanks to the quasi-minimality property proved in the previous proposition
and to the fact that $\partial U$ and $\partial\Omega$ meet orthogonally,
we can apply Theorem~\ref{teo:dlb} to deduce that there exist constants $\vartheta_0>0$, $\rho_0>0$ (independent of $n$) such that
for every $x\in \overline{S}_{w_n}\cap\ombar'$ and for every $\rho\leq\rho_0$
\begin{equation} \label{lowerbound}
\hu(S_{w_n}\cap B_\rho(x)) \geq \vartheta_0\rho.
\end{equation}
The properties in the statement are standard consequences of \eqref{lowerbound} (see \cite{AFP}).
\end{proof}

Corollary~\ref{cor:kuratowsky} provides the Hausdorff convergence of $\overline{S}_{w_n}$ to $\overline{S}_u$ in $\ombar'$,
which allows us to assume, from now on, that $\overline{S}_{w_n}$ is contained in a tubular neighborhood of $\overline{S}_u$ contained in $U$.
We now come to the main consequence of the regularity theory for quasi-minimizers.
We follow here the notation introduced in Section~\ref{sect:prelparreg}.

We first observe that, using the good description of $\overline{S}_{w_n}$ near $\partial\Omega$ given by Theorem~\ref{teo:parreg2},
we can find $\tau>0$ such that $\overline{S}_{w_n}\cap\Omega(\tau)$ is a $C^{1,\alpha}$-curve for some $\alpha\in(0,1)$,
with $C^{1,\alpha}$-norms uniformly bounded with respect to $n$ and meeting $\partial\Omega$ orthogonally.
Combining this information with the Hausdorff convergence to $\overline{S}_u$,
we deduce that the sets $\overline{S}_{w_n}$ converge to $\overline{S}_u$ in $\Omega(\tau)$
in the $C^{1,\beta}$-sense, for every $\beta<\alpha$.
In the following proposition we obtain the same convergence in the interior of $\Omega$.

\begin{proposition} \label{prop:david}
There exists a finite covering of $\Gamma\cap(\Omega\setminus\Omega(\tau))$ of the form
$\bigcup_{i=1}^{N_0} (x_i + C_{\nu_i,\rho_i})$
where $x_i\in\Gamma$, $\nu_i=\nu_\Gamma(x_i)$,
and functions $f_i^{(n)}:(-\rho_i,\rho_i)\to(-\rho_i,\rho_i)$ of class $C^{1,\alpha}$ (for some $\alpha\in(0,1)$)
such that
$$
(\overline{S}_{w_n}-x_i)\cap C_{\nu_i,\rho_i} = \gr_{\nu_i}(f^{(n)}_i)
$$
for $n$ sufficiently large and $i=1,\ldots,N_0$.
Moreover, the sequence $f_i^{(n)}$ converges to $f_i$ in $C^{1,\beta}$ as $n\to+\infty$ for every $\beta<\alpha$,
where $f_i:(-\rho_i,\rho_i)\to(-\rho_i,\rho_i)$ is such that
$$
(\Gamma-x_i)\cap C_{\nu_i,\rho_i} = \gr_{\nu_i}(f_i).
$$
\end{proposition}

\begin{proof}
Fix any point $x_0\in \Gamma\cap(\Omega\setminus\Omega(\tau))$. By the regularity of $u$ and $\Gamma=\overline{S}_u$,
we can find $r_0>0$ such that $B_{r_0}(x_0)\subset\Omega\cap U$, $\hu(S_u\cap\partial B_{r_0}(x_0))=0$ and
$$
E_u(x_0,r_0) < \e_0\frac{r_0}{8}\,,
$$
where $\e_0$ is given by Theorem~\ref{teo:parreg}.
Lemma~\ref{lemma:convergenzeloc} immediately implies that
$D_{w_n}(x_0,r_0) \to D_u(x_0,r_0)$
and that for every affine plane $T$
$$
\int_{\overline{S}_{w_n}\cap B_{r_0}(x_0)} \dist^2(y,T)\,d\hu(y) \to \int_{\overline{S}_{u}\cap B_{r_0}(x_0)} \dist^2(y,T)\,d\hu(y) \,.
$$
From the previous convergence it follows also that $\limsup_{n\to\infty}A_{w_n}(x_0,r_0)\leq A_u(x_0,r_0)$,
since if the minimum value defining $A_u(x_0,r_0)$ is attained at an affine plane $T_0$, then
$$
A_{w_n}(x_0,r_0)\leq\int_{\overline{S}_{w_n}\cap B_{r_0}(x_0)}\dist^2(y,T_0)\,d\hu(y)
\to \int_{\overline{S}_u\cap B_{r_0}(x_0)}\dist^2(y,T_0)\,d\hu(y) = A_u(x_0,r_0)\,.
$$
Hence $\limsup_{n\to\infty}E_{w_n}(x_0,r_0) \leq E_u(x_0,r_0)$, so that for $n$ sufficiently large we have
$$
E_{w_n}(x_0,r_0) < \e_0\frac{r_0}{8}\,.
$$
By Corollary~\ref{cor:kuratowsky} we can find a sequence $x_n\in \overline{S}_{w_n}$ converging to $x_0$,
so that $B_{r_0/2}(x_n)\subset B_{r_0}(x_0)$ for $n$ large enough and thus
$$
E_{w_n}(x_n,r_0/2) = D_{w_n}(x_n,r_0/2) + \frac{4}{r_0^2} A_{w_n}(x_n,r_0/2) \leq 4 E_{w_n}(x_0,r_0) <\e_0\frac{r_0}{2}.
$$
We can then apply Theorem~\ref{teo:parreg}:
we find a radius $r_1\in(0,r_0)$ and functions $g_n:(-r_1,r_1)\to\R$ uniformly bounded in $C^{1,\frac14}$,
with $g_n(0)=g_n'(0)=0$,
such that $(\overline{S}_{w_n}-x_n)\cap C_{\nu_n,r_1}=\gr_{\nu_n}(g_n)$, where $\nu_n$ is the normal to $\overline{S}_{w_n}$ at $x_n$.

By compactness, $\nu_n\to\bar{\nu}$ (up to subsequences).
For $n$ large enough $C_{\bar{\nu},r_1/2}\subset C_{\nu_n,r_1}+x_n-x_0$,
and there exist functions $f_n$ uniformly bounded in $C^{1,\frac14}$ such that
$\gr_{\bar{\nu}}(f_n)\cap C_{\bar{\nu},r_1/2} = (\gr_{\nu_n}(g_n)+x_n-x_0)\cap C_{\bar{\nu},r_1/2}$.
Hence
$$
(\overline{S}_{w_n}-x_0)\cap C_{\bar{\nu},r_1/2} = \gr_{\bar{\nu}}(f_n),
$$
and by Ascoli-Arzel\`{a} Theorem $f_n$ converges to some function $f$ in $C^{1,\beta}$ for every $\beta<\frac14$,
with $f(0)=f'(0)=0$.
Using the Kuratowski convergence of $\overline{S}_{w_n}$ to $\Gamma$, we deduce that $(\Gamma-x_0)\cap C_{\bar{\nu},r_1/2}=\gr_{\bar{\nu}}(f)$,
and since $f'(0)=0$ it must be $\bar{\nu}=\nu_\Gamma(x_0)$.
\end{proof}

From what we have proved it follows that for every $n\in\N$ there exists a diffeomorphism $\Phi_n:\ombar\to\ombar$,
with $\supp(\Phi_n-Id)\subset\subset (U\setminus\partial_D\Omega)$, such that $\overline{S}_{w_n}=\Phi_n(\Gamma)$ and $\|\Phi_n-Id\|_{C^{1,\alpha}(\Gamma)}\to0$.

With this information, we can finally conclude the proof of the isolated local minimality of $u$.
Indeed, since $\hu(\overline{S}_{w_n}\setminus S_{w_n})=0$ by Corollary~\ref{cor:kuratowsky},
we have that $(\Phi_n(\Gamma),w_n)\in\spazio$ and $F(\Phi_n(\Gamma),w_n)=\F(w_n)$.
Hence for $n$ large enough, using \eqref{contraminloc2},
\begin{align*}
F(\Phi_n(\Gamma),w_n)
= \F(w_n)
\leq \F(u)
= F(\Gamma,u) \,,
\end{align*}
which implies that $\Phi_n=Id$ and $w_n=u$ for all (large) $n$
by Proposition~\ref{prop:minw1}.
Hence $u$ itself is a solution to $\eqref{penalizzaton}$,
and as a consequence of \eqref{contraminloc} also $v_n$ solves the same minimum problem.
We can then repeat all the previous argument for the sequence $v_n$ instead of $w_n$,
which leads, as before, to $v_n=u$ for $n$ sufficiently large.
This is the desired contradiction, since we are assuming $v_n\neq u$ for every $n$.

\end{section}


\begin{section}{Applications and examples} \label{sect:example}

We start this section by showing that any regular critical pair $(\Gamma,u)$ satisfying \eqref{hpps} is strictly stable
in a sufficiently small tubular neighborhood $\mathcal{N}_\e(\Gamma)$ of the discontinuity set.
As a consequence of our main result, we deduce the local minimality of $(\Gamma,u)$ in $\mathcal{N}_\e(\Gamma)$,
and also that $(\Gamma,u)$ is in fact a global minimizer in a smaller neighborhood.
This is in analogy with the result proved in \cite{MM},
where it is shown, by means of a calibration method,
that a critical point is a Dirichlet minimizer in small domains.

\begin{proposition}[local and global minimality in small tubular neighborhoods] \label{prop:smalldom}
Let $(\Gamma,u)$ be a regular critical pair satisfying condition \eqref{hpps}.
Then there exists $\e_0>0$ such that the tubular neighborhood $\mathcal{N}_\e(\Gamma)$ of $\Gamma$
is an admissible subdomain and $(\Gamma,u)$ is strictly stable in $\mathcal{N}_\e(\Gamma)$ for every $\e<\e_0$.
In particular, there exists $\delta>0$ such that $F(\Gamma,u)<F(K,v)$
for every $(K,v)\in\mathcal{A}(\Omega)$ with $0<\|u-v\|_{L^1(\Omega)}<\delta$
and $v=u$ in $\Omega\setminus\mathcal{N}_{\e}(\Gamma)$.

Moreover, there exists $\e_1\in(0,\e_0)$ such that $(\Gamma,u)$ is a global minimizer in $\mathcal{N}_\e(\Gamma)$ for every $\e<\e_1$,
in the sense that $F(\Gamma,u)\leq F(K,v)$ for every $(K,v)\in\spazio$ with $v=u$ in $\Omega\setminus\mathcal{N}_\e(\Gamma)$.
\end{proposition}

\begin{proof}
Clearly $\mathcal{N}_\e(\Gamma)$ is an admissible subdomain for $\e$ small enough, and in view of Proposition~\ref{prop:T}
we shall prove that
$$
\lim_{\e\to0}\mu(\mathcal{N}_\e(\Gamma))=+\infty
$$
in order to obtain the first part of the statement.
Assume by contradiction that there exist $\e_n\to0^+$, $C>0$ and $v_n\in H^1_{U_n}(\Omega\setminus\Gamma)$
such that $\|\Phi_{v_n}\|_{\sim}=1$ and
$$
2\int_\Omega|\nabla v_n|^2\leq C
$$
for every $n$, where we set $U_n:=\mathcal{N}_{\e_n}(\Gamma)$.
Then $v_n$ is a bounded sequence in $H^1_{U_1}(\Omega\setminus\Gamma)$,
which converges weakly to 0 since the measure of $U_n$ goes to 0.
By compactness of the map $v\mapsto\Phi_v$, we have that $\Phi_{v_n}$ converge to 0 strongly in $H^1(\Gamma\cap\Omega)$,
which is in contradiction to the fact that $\|\Phi_{v_n}\|_\sim=1$ for every $n$.

To prove the second part of the statement,
let $u_\e$ be a solution to the minimum problem
\begin{equation} \label{eq:minglob}
\min\bigl\{ \F(v) : v\in\sbv(\Omega), \; v=u\text{ in }\Omega\setminus \mathcal{N}_\e(\Gamma) \bigr\},
\end{equation}
where $\F$ is the relaxed functional introduced at the beginning of Section~\ref{sect:minSBV}.
We remark that, by classical regularity results for minimizers of the Mumford-Shah functional,
$\hu(\overline{S}_{u_\e}\setminus S_{u_\e})=0$ and thus $\F(u_\e)=F(\overline{S}_{u_\e},u_\e)$.
Hence, since $u_\e\to u$ in $L^1(\Omega)$ as $\e\to0$ because the measure of $\mathcal{N}_\e(\Gamma)$ goes to 0,
we conclude that $u_\e=u$ for $\e$ small enough, as a consequence of the isolated local minimality of $(\Gamma,u)$.
Then $u$ is a solution to \eqref{eq:minglob}, and the conclusion follows by Remark~\ref{rm:sbv}.
\end{proof}

\begin{remark}
Let $(\Gamma,u)$ be a regular critical pair,
and assume that
\begin{align*}
-2\int_\Omega |\nabla \vf|^2\,dx
+\int_{\Gamma\cap\Omega} |\nabla_\Gamma\vphi|^2\,d\hu
+\int_{\Gamma\cap\Omega} H^2\vphi^2\,d\hu
-\int_{\Gamma\cap\partial\Omega} H_{\partial\Omega}\vphi^2\,d\huu>0
\end{align*}
for every $\vphi\in H^1(\Gamma\cap\Omega)\setminus\{0\}$,
where $\vf\in H^1(\Omega\setminus\Gamma)$, $\vf=0$ on $\partial_D\Omega$, solves
\begin{equation*}
\int_{\Omega} \nabla\vf \cdot \nabla z\,dx
+\int_{\Gamma\cap\Omega} \bigl[ z^+\div_\Gamma(\vphi\nabla_\Gamma u^+) - z^-\div_\Gamma(\vphi\nabla_\Gamma u^-) \bigr] \,d\hu =0
\end{equation*}
for every $z\in H^1(\Omega\setminus\Gamma)$ with $z=0$ on $\partial_D\Omega$.
Then $(\Gamma,u)$ is strictly stable in every admissible subdomain $U$.
Hence, under the previous assumptions we can conclude that for every neighborhood $\mathcal{N}_\eta(\mathcal{S})$,
where $\mathcal{S}$ is the relative boundary of $\partial_D\Omega$ in $\partial\Omega$,
there exists $\delta(\eta)>0$ such that $F(\Gamma,u)<F(K,v)$
for every $(K,v)\in\spazio$ with $\|v-u\|_{L^1(\Omega)}<\delta$ and $v=u$ in $\mathcal{N}_\eta(\mathcal{S})$.
\end{remark}

We now provide some explicit examples of critical point to which Theorem~\ref{teo:minSBV} can be applied.
In particular, in Example~\ref{ex:1} we discuss how the stability of constant critical pairs
depends on the geometry of the domain $\Omega$,
while in Remark~\ref{ex:2} we discuss how to construct families of (non-constant) critical pairs
by a perturbing the Dirichlet data.

\begin{example} \label{ex:1}
Let $\Gamma$ be a straight line contained in $\Omega$ connecting two points $x_1, x_2\in\partial\Omega$ of minimal distance,
and let $u$ be equal to two different constants in the two connected components of $\Omega\setminus \Gamma$.
Assume that $\Omega$ is strictly concave at $x_1$ and $x_2$
(that is, the curvature $H_{\partial\Omega}$ with respect to the exterior normal is strictly negative at $x_1$ and $x_2$).
Then $(\Gamma,u)$ is a regular critical pair such that for every admissible subdomain $U$
$$
\partial^2F((\Gamma,u);U)[\vphi] = \int_\Gamma|\nabla_\Gamma\vphi|^2 - H_{\partial\Omega}(x_1)\vphi^2(x_1) - H_{\partial\Omega}(x_2)\vphi^2(x_2)>0
$$
for every $\vphi\in H^1(\Gamma)\setminus\{0\}$.
Hence it follows by Theorem~\ref{teo:minSBV} that $(\Gamma,u)$ is an isolated local minimizer for $F$
in every admissible subdomain $U$.

If the domain $\Omega$ is strictly convex, then a straight line connecting two points on $\partial\Omega$ of minimal distance
is never a local minimizer: indeed, if $U$ is any admissible subdomain,
by evaluating the quadratic form $\partial^2F((\Gamma,u);U)$ at the constant function $\vphi=1$ we get
$$
\partial^2F((\Gamma,u);U)[1] = -2 \int_\Omega|\nabla v_\vphi|^2 - H_{\partial\Omega}(x_1) - H_{\partial\Omega}(x_2)<0.
$$
We remark that this is not in contradiction to the result of Proposition~\ref{prop:smalldom},
since in the present situation condition \eqref{hpps} is not satisfied.
\end{example}

\begin{remark}[families of stable critical pairs by perturbation of the Dirichlet data] \label{ex:2}
Let $(\Gamma,u)$ be a strictly stable regular critical pair in an admissible subdomain $U$,
and assume in addition that $u^+$ and $u^-$ are of class $C^{1,\alpha}$ in a neighborhood of $\Gamma$.

We fix a function $\psi_0\in C^\infty_c(\partial_D\Omega)$
and we consider the perturbed Dirichlet datum $u_\e:=u+\e\psi_0$ for $\e>0$.
As an application of the Implicit Function Theorem, one can show that for every $\e$ sufficiently small
there exists a strictly stable regular critical pair $(\Gamma_\e,v_\e)$
with $v_\e=u_\e$ in $(\Omega\setminus U)\cup\partial_D\Omega$.

The idea of the proof is to associate, with every $\psi\in C^{2,\alpha}(\Gamma)$, the curve
$\Gamma_\psi$ defined as in \eqref{eq:gammapsi}
and the function $u_{\e,\psi}$ which minimizes the Dirichlet integral in $H^1(\Omega\setminus\Gamma_\psi)$
and attains the boundary condition $u_{\e,\psi}=u_\e$ in $(\Omega\setminus U)\cup\partial_D\Omega$.
Then one can prove that the map
$$
G:\R\times C^{2,\alpha}(\Gamma) \to C^{0,\alpha}(\Gamma),
\qquad
G(\e,\psi) := H_\psi - |\nabla_{\Gamma_\psi}u_{\e,\psi}^+|^2 + |\nabla_{\Gamma_\psi}u_{\e,\psi}^-|^2
$$
(where $H_\psi$ denotes the curvature of $\Gamma_\psi$)
is of class $C^1$ in a neighborhood of $(0,0)$,
satisfies $G(0,0)=0$ (as $(\Gamma,u)$ is a critical pair),
and the partial derivative $\partial_\psi G(0,0)$ is an invertible bounded linear operator,
thanks to the strict positivity of the second variation at $(\Gamma,u)$.
Hence it is possible to apply the Implicit Function Theorem and to obtain the desired family of critical pairs.
\end{remark}

We conclude this section by observing, in the following remark,
that our analysis can be extended to the periodic case:
more precisely, we assume that the domain is a rectangle,
$\Gamma$ is a curve joining two opposite points on the boundary,
and the Neumann boundary conditions are replaced by periodicity conditions on the sides connected by $\Gamma$.
The remaining pair of sides represents the Dirichlet part of the boundary.
We also discuss an explicit example in this different setting.
In the remaining part of this section, with a slight abuse of notation
we denote the generic point of $\R^2$ by $(x,y)$.

\begin{remark} \label{rm:periodic}
Let $R:=[0,b)\times(-a,a)$, where $a,b>0$ are positive real numbers.
We define the infinite strip $\widetilde{R}:=\R\times(-a,a)$,
the Dirichlet boundary $\partial_DR:=[0,b]\times\{-a,a\}$,
and the class of admissible pairs
\begin{align*}
\mathcal{A}(R):=\bigl\{(K,v): K\subset\R^2 \text{ closed, }\, &K+(b,0)=K,\;
v\in H^1_{\rm loc}(\widetilde{R}\setminus K)\cap H^1(R\setminus K), \\
& v_x(x+b,y)=v_x(x,y) \text{ for every }(x,y)\in\widetilde{R}\setminus K \bigr\}\,.
\end{align*}
We denote by $H^1_{\rm per}(R\setminus K)$ the class of functions
$z\in H^1_{\rm loc}(\widetilde{R}\setminus K)\cap H^1(R\setminus K)$
such that the map $x\mapsto z(x,y)$ is $b$-periodic for every $y\in(-a,a)$.
Finally we consider the functional
$$
F(K,v) := \int_{R\setminus K} |\nabla v|^2 + \hu(K\cap R)
\qquad\text{for }(K,v)\in\mathcal{A}(R).
$$
Similarly to what we did in Section~\ref{sect:setting}, we say that $(\Gamma,u)\in\mathcal{A}(R)$
is a regular critical pair if $\Gamma\subset\widetilde{R}$ is a curve of class $C^\infty$
such that $\Gamma\cap R$ connects two opposite points on the $\partial R$, $u$ satisfies
$$
\int_{R\setminus\Gamma} \nabla u\cdot\nabla z =0
\qquad\text{for every }z\in H^1_{\rm per}(R\setminus\Gamma)\text{ with }z=0\text{ on }\partial_DR,
$$
and moreover the transmission condition and the non-vanishing jump condition (see Definition~\ref{def:critpair}) hold on $\Gamma$.
Setting $H^1_{\rm per}(\Gamma):=\{\vphi\in H^1_{\rm loc}(\Gamma) : \vphi(x+b,y)=\vphi(x,y) \text{ for every }(x,y)\in\Gamma\}$,
we say that a regular critical pair $(\Gamma,u)$ is strictly stable if
$$
\partial^2F(\Gamma,u)[\vphi] := -2\int_R |\nabla v_{\vphi}|^2 + \int_{\Gamma\cap R}|\nabla_\Gamma\vphi|^2\,d\hu + \int_{\Gamma\cap R}H^2\vphi^2\,d\hu >0
$$
for every $\vphi\in H^1_{\rm per}(\Gamma)\setminus\{0\}$,
where $v_\vphi\in H^1_{\rm per}(R\setminus\Gamma)$, $v_{\vphi}=0$ on $\partial_DR$, is the solution to
\begin{equation}\label{esempio2}
\int_R \nabla v_{\vphi}\cdot\nabla z + \int_{\Gamma\cap R} \Bigl[ z^+\div_\Gamma\bigl(\vphi\nabla_\Gamma u^+\bigr) - z^-\div_\Gamma\bigl(\vphi\nabla_\Gamma u^-\bigr) \Bigr]\,d\hu =0
\end{equation}
for every $z\in H^1_{\rm per}(R\setminus\Gamma)$, $z=0$ on $\partial_DR$.

Then one can prove that every strictly stable regular critical pair $(\Gamma,u)$ is a local minimizer,
in the sense that there exists $\delta>0$ such that
$F(\Gamma,u)<F(K,v)$ for every $(K,v)\in\mathcal{A}(R)$ with $v=u$ on $\partial_DR$ and $0<\|u-v\|_{L^1(R)}<\delta$.
We omit the proof of this result, since it can be obtained by repeating all the arguments which lead to the proof of Theorem~\ref{teo:minSBV} with the natural modifications (notice that the proof in the present setting is in fact simpler, since by periodicity we can work in the whole strip $\widetilde{R}$ avoiding the technical difficulties related to the presence of Neumann boundary conditions).
\end{remark}

\begin{example} \label{ex:3}
Here we adapt to the periodic setting described in Remark~\ref{rm:periodic}
the example discussed in \cite[Section~7]{CMM}. Setting $R=[0,b)\times(-a,a)$,
we consider the regular critical pair $(\Gamma,u)\in\mathcal{A}(R)$ where $\Gamma=\R\times\{0\}$
and $u:\R^2\to\R$ is the function
$$
u(x,y):=
\left\{
  \begin{array}{ll}
    x+1 & \hbox{for }y\geq0, \\
    -x & \hbox{for }y<0.
  \end{array}
\right.
$$
Notice that the energy of $(\Gamma,u)$ is invariant along vertical translations of the discontinuity set.
Nevertheless, we shall prove in fact that if
\begin{equation} \label{esempio0}
{\textstyle \frac{2b}{\pi}} \tanh\bigl({\textstyle \frac{2\pi a}{b}}\bigr) <1\,,
\end{equation}
then $(\Gamma,u)$ is an isolated local minimizer up to vertical translations:
precisely, there exists $\delta>0$ such that $F(\Gamma,u)<F(K,v)$ for every $(K,v)\in\mathcal{A}(R)$
with $v=u$ on $\partial_DR$ and $\|u-v\|_{L^1(R)}<\delta$,
unless $K$ coincides with a vertical translation of $\Gamma$.
Moreover, \eqref{esempio0} is sharp in the sense that if
${\textstyle \frac{2b}{\pi}} \tanh\bigl({\textstyle \frac{2\pi a}{b}}\bigr) > 1$
then $(\Gamma,u)$ is unstable.

To this aim, we will test the strict positivity of second variation at $(\Gamma,u)$
on the subspace ${H}^1_0(0,b)$ of $H^1_{\rm per}(\Gamma)$ of the functions vanishing at the endpoints,
showing that
\begin{equation}\label{esempio1}
\partial^2F(\Gamma,u)[\vphi]\geq C_0\|\vphi\|_{H^1(0,b)}
\text{ for every }\vphi\in{H}^1_0(0,b)\setminus\{0\}
\quad\text{iff}\quad
{\textstyle \frac{2b}{\pi}} \tanh\bigl({\textstyle \frac{2\pi a}{b}}\bigr) <1\,.
\end{equation}
In turn, setting $\Gamma_\e:=\R\times\{\e\}$ and
$$
u_\e(x,y):=
\left\{
  \begin{array}{ll}
    x+1 & \hbox{for }y\geq\e, \\
    -x & \hbox{for }y<\e,
  \end{array}
\right.
$$
we have that $(\Gamma_\e,u_\e)$ is still a critical pair with the same energy of $(\Gamma,u)$,
and, assuming \eqref{esempio0} and \eqref{esempio1},
there exists $\e_0>0$ such that for every $\e\in(-\e_0,\e_0)$ we have
\begin{equation}\label{esempio3}
\partial^2F(\Gamma_\e,u_\e)[\vphi]\geq \frac{C_0}{2}\|\vphi\|_{H^1(0,b)}^2
\qquad\text{for every }\vphi\in{H}^1_0(0,b)\setminus\{0\}.
\end{equation}
This can be deduced by comparing the explicit expressions of the second variation at $(\Gamma,u)$ and at $(\Gamma_\e,u_\e)$ and observing that
$$
\sup_{\|\vphi\|_{H^1(0,b)}=1} \bigg| \int_R|\nabla v_\vphi^\e|^2 - \int_R |\nabla v_\vphi|^2 \bigg| \to 0
\qquad\text{as }\e\to0
$$
(where $v_\vphi$ and $v_\vphi^\e$ are the solutions to \eqref{esempio2}
corresponding to $(\Gamma,u)$ and $(\Gamma_\e,u_\e)$ respectively);
this last estimate is obtained by subtracting the equations satisfied by $v_\vphi$ and $v_\vphi^\e$.
From \eqref{esempio3} it follows that
any configuration which is close in $W^{2,\infty}$ and coincides with $\Gamma_\e$ at the endpoints
has strictly larger energy than $(\Gamma_\e,u_\e)$:
more precisely, there exists $\delta_0>0$ such that for every $|\e|<\e_0$,
for every $b$-periodic function $h\in W_{\rm loc}^{2,\infty}(\R)$ with $0<\|h-\e\|_{W^{2,\infty}(0,b)}<\delta_0$, $h(0)=h(b)=\e$,
and for every $v$ such that $(\Gamma_h,v)\in\mathcal{A}(R)$ and $v=u$ on $\partial_DR$,
we have $F(\Gamma_h,v)>F(\Gamma_\e,u_\e)=F(\Gamma,u)$,
where we denoted by $\Gamma_h$ the graph of $h$.
This can be deduced by repeating the arguments for the proof of Theorem~\ref{teo:minW},
paying attention to the fact that the local minimality neighborhood can be chosen uniform with respect to $n$.

In turn, from this property easily follows the isolated local $W^{2,\infty}$-minimality of $(\Gamma,u)$,
since it implies the existence of a positive $\delta$ such that for every $(\Gamma_h,v)\in\mathcal{A}(R)$
with $0<\|h\|_{W^{2,\infty}(0,b)}<\delta$ and $v=u$ on $\partial_DR$ we have $F(\Gamma_h,v)>F(\Gamma,u)$,
unless $\Gamma_h=\Gamma_\e$ for some $\e>0$ and $v=u_\e$.
Finally, this property implies also the local $L^1$-minimality (up to translations), by the same argument developed in Sections~\ref{sect:minW1inf} and \ref{sect:minSBV}.

\smallskip
We are left with the proof of \eqref{esempio1}.
Condition \eqref{hpps} is automatically satisfied on the subspace ${H}^1_0(0,b)$,
and we can discuss the sign of $\partial^2F(\Gamma,u)$ in terms of the eigenvalue $\lambda_1$ introduced in \eqref{lamda1}.
We will prove that
\begin{equation}\label{eq:example}
\lambda_1 (R)= \frac{2b}{\pi}\tanh\frac{2\pi a}{b}\:.
\end{equation}
We remark that $\lambda_1$ coincides with the greatest $\lambda$
such that there exists a nontrivial solution $(v,\vphi)\in H^1_{\rm per}(R\setminus\Gamma)\times \widetilde{H}^1_{0}(0,b)$,
$v=0$ in $\partial_DR$,
to the equations
$$
\lambda\int_R\nabla v\cdot\nabla z + \int_0^b\bigl(\vphi'z^++\vphi'z^-\bigr)\,dx=0,
\qquad
\int_0^b\bigl( \vphi'\psi' + 2 \psi'v^+ + 2 \psi' v^- \bigr)\,dx=0
$$
for every $z\in H^1_{\rm per}(R\setminus\Gamma)$ with $z=0$ on $\partial_DR$,
and for every $\psi\in \widetilde{H}^1_{0}(0,b)$.
By symmetry, $v(x,y)=v(x,-y)$, so that by setting $R^+:=(0,b)\times(0,a)$, we look for a solution to
$$
\left\{
  \begin{array}{ll}
    \Delta v=0 & \hbox{in }R^+, \\
    v=0 & \hbox{on }\partial_DR, \\
    \lambda\partial_y v=\vphi' & \hbox{on }\Gamma, \\
    \vphi''=-4\partial_xv & \hbox{on }\Gamma.
  \end{array}
\right.
$$
The last two conditions say that
$$
\lambda\partial_yv(x,0)=-4\bigl(v(x,0)-c\bigr), \qquad c:=\frac{1}{b}\int_0^bv(x,0)\,dx\,.
$$
Hence we are left with the determination of the greatest $\lambda$ such that
there exists a nontrivial periodic solution $v$ to the system
$$
\left\{
  \begin{array}{ll}
    \Delta v=0 & \hbox{in }R^+, \\
    v=0 & \hbox{on }\partial_DR, \\
    \lambda\partial_y v=-4\bigl(v-c\bigr) & \hbox{on }\Gamma.
  \end{array}
\right.
$$
We expand $v(\cdot,y)$ in series of cosines:
$$
v(x,y) = \sum_{n=0}^{+\infty} c_n(y)\cos\bigl({\textstyle \frac{n\pi}{b}}x\bigr),
$$
and by the first two condition of the system we have that $c_n(y)=c_n\sinh\bigl({\textstyle \frac{n\pi}{b}}(a-y)\bigr)$, with $c_n\in\R$.
Hence
$$
v(x,y) = \sum_{n=0}^{+\infty} c_n\cos\bigl({\textstyle \frac{n\pi}{b}}x\bigr)\sinh\bigl({\textstyle \frac{n\pi}{b}}(a-y)\bigr)
$$
and by imposing the last condition of the system we have
\begin{align*}
\lambda \sum_{n=0}^{+\infty} c_n {\textstyle \frac{n\pi}{b}} \cos\bigl({\textstyle \frac{n\pi}{b}}x\bigr)\cosh\bigl({\textstyle \frac{n\pi}{b}}a\bigr)
= 4 \sum_{n=0}^{+\infty} c_n \cos\bigl({\textstyle \frac{n\pi}{b}}x\bigr)\sinh\bigl({\textstyle \frac{n\pi}{b}}a\bigr) -4c\,.
\end{align*}
By expanding also $c$ in series of cosines, we deduce from the previous inequality that $c=0$, and also
$$
\lambda c_n{\textstyle \frac{n\pi}{b}} \cosh\bigl({\textstyle \frac{n\pi a}{b}}\bigr)
= 4 c_n \sinh\bigl({\textstyle \frac{n\pi a}{b}}\bigr)
$$
for all $n\geq 1$. Hence, since we are looking for a positive $\lambda$,
it follows that $\lambda = {\textstyle \frac{4b}{n\pi}} \tanh\bigl({\textstyle \frac{n\pi a}{b}}\bigr)$
whenever $c_n\neq0$.
Thus only one of the coefficients $c_n$ can be different from 0,
and by periodicity it must correspond to an even index
(here we used also the fact that the function
$t\mapsto\frac{4b}{t\pi} \tanh\bigl({\textstyle \frac{t\pi a}{b}}\bigr)$
is monotone decreasing).
Hence there exists $\bar{n}\geq2$ even such that $c_{\bar{n}}\neq0$ and
$$
\lambda = {\textstyle \frac{4b}{\bar{n}\pi}} \tanh\bigl({\textstyle \frac{\bar{n}\pi a}{b}}\bigr)\,,
$$
and clearly the largest value of $\lambda$ corresponds to $\bar{n}=2$.
This completes the proof of \eqref{eq:example} and, in turn, of \eqref{esempio1}.
\end{example}

\end{section}


\begin{section}{Appendix} \label{sect:appendix}

We collect here some technical results which have been used in the paper.
In the following lemma we assume to be in the same setting as described at the beginning of Section~\ref{sect:minW}.

\begin{lemma} \label{lemma:grafici}
Let $(\Gamma_n)_n$ be a sequence of curves of class $C^{1,\alpha}$, for some $\alpha\in(0,1)$,
converging to $\Gamma$ in $C^{1,\alpha}$,
in the sense that there exist diffeomorphisms $\Phi_n:\ombar\to\ombar$ of class $C^{1,\alpha}$ such that
$\Gamma_n=\Phi_n(\Gamma)$ and $\|\Phi_n-Id\|_{C^{1,\alpha}(\Gamma)}\to0$.

Then there exist $\psi_n\in C^{1,\alpha}(\Gamma)$, with $\psi_n\to0$ in $C^{1,\alpha}(\Gamma)$,
such that $\Gamma_n=\Gamma_{\psi_n}$,
where $\Gamma_{\psi_n}$ is the set defined according to \eqref{eq:gammapsi}.

Moreover, denoting by $H_{\Gamma_n}$ and $H$ the curvatures of $\Gamma_n$ and of $\Gamma$ respectively, if
\begin{equation} \label{convcurv}
\| H_{\Gamma_n}\circ\Phi_n - H \|_{L^\infty(\Gamma)}\to0
\end{equation}
then $\psi_n$ is of class $W^{2,\infty}$ and $\psi_n\to0$ in $W^{2,\infty}(\Gamma)$.
\end{lemma}

\begin{proof}
We first extend each curve $\Gamma_n$ (and $\Gamma$ itself) outside $\ombar$ as a straight line
so that the resulting curves are of class $C^{1,\alpha}$ and still converge to $\Gamma$ in the $C^{1,\alpha}$ sense.
We can then localize in a small square $R=(-\rho,\rho)\times(-\rho,\rho)$
(which we assume for simplicity centered at the origin)
in which we can express $\Gamma$ and $\Gamma_n$ as graphs of $C^{1,\alpha}$ functions:
$$
\Gamma_n \cap R = \{(x,f_n(x)):x\in(-\rho,\rho)\},
\qquad
\Gamma \cap R = \{(x,f(x)):x\in(-\rho,\rho)\}
$$
with $f_n\to f$ in $C^{1,\alpha}$.
By a covering argument it is sufficient to prove the result in $R$
(notice that, by our extension of the curves outside $\ombar$,
in this way we can cover also a neighborhood of the intersection of $\Gamma$ with $\partial\Omega$).

We recall that in a sufficiently small tubular neighborhood $\mathcal{N}_{\eta_0}(\Gamma)$ of $\Gamma$
are well defined two maps $\pi:\mathcal{N}_{\eta_0}(\Gamma)\to\Gamma$, $\tau:\mathcal{N}_{\eta_0}(\Gamma)\to\R$
of class $C^2$ (thank to the $C^2$ regularity of the vector field $X$ generating the flow $\Psi$)
such that $y=\Psi(\tau(y),\pi(y))$ for every $y$.

Taking $\rho'<\rho$, for $n$ sufficiently large we can define a map $\tilde{\pi}_n:(-\rho',\rho')\to(-\rho,\rho)$
by setting $\tilde{\pi}_n(x):=\pi_1\circ\pi(x,f_n(x))$, where $\pi_1(x,y):=x$.
Notice that $\tilde{\pi}_n$ tends to the identity in $C^{1,\alpha}$,
hence it is invertible and also its inverse converges to the identity in $C^{1,\alpha}$.
Defining
$$
\phi_n(x) := \tau \bigl(\tilde{\pi}_n^{-1}(x),f_n(\tilde{\pi}_n^{-1}(x)) \bigr)
$$
for $x\in(-\rho',\rho')$,
since $\tau$ is regular and vanishes on $\Gamma$ we deduce that $\phi_n\to0$ in $C^{1,\alpha}(-\rho',\rho')$.

Hence the map $\psi_n(x,f(x)):=\phi_n(x)$, for $|x|<\rho'$, is of class $C^{1,\alpha}$ on $\Gamma\cap((-\rho',\rho')\times(-\rho,\rho))$, converges to 0 in $C^{1,\alpha}$ and satisfies $\Gamma_{\psi_n}=\Gamma_n$.
This proves the first part of the statement.

The second part follows similarly:
indeed, since the sets $\Gamma_n$ are locally one-dimensional graphs,
the boundedness in $L^\infty$ of the curvatures of $\Gamma_n$ yields the $W^{2,\infty}$-regularity of the functions $f_n$,
and the convergence \eqref{convcurv} implies in addition that $f_n\to f$ in $W^{2,\infty}$.
Hence the conclusion follows from the explicit expression of $\psi_n$ obtained above.
\end{proof}

We conclude with two regularity results for the Neumann problem and for the mixed Dirichlet-Neumann problem
in planar domains with angles.

\begin{lemma} \label{lemma:Neumannreg}
Let $A$ be an open subset of the unit ball $B_1$
such that $\partial A\cap B_1 = \Gamma_1 \cup \Gamma_2$,
where $\Gamma_1$ and $\Gamma_2$ are two curves of class $C^{1,\beta}$
meeting at the origin with an internal angle $\alpha\in(0,\pi)$.
Let $u\in H^1(A)$ be a weak solution to
$$
\left\{
  \begin{array}{ll}
    \Delta u=0 & \hbox{in }A, \\
    \partial_\nu u=0 & \hbox{on }\Gamma_1\cup\Gamma_2.
  \end{array}
\right.
$$
Then $\nabla u$ has a $C^{0,\gamma}$ extension up to $\Gamma_1\cup\Gamma_2$, for $\gamma=\min\{\beta,\frac{\pi}{\alpha}-1\}$,
with $C^{0,\gamma}$-norm bounded by a constant depending only on the $C^{1,\beta}$-norm of $\Gamma_1$ and $\Gamma_2$.
\end{lemma}

\begin{proof}
We consider $A$ as a subset of the complex plane $\C$
(we can assume without loss of generality that the positive real axis coincides with the tangent to $\Gamma_1$ at the origin,
and that the tangent to $\Gamma_2$ at the origin is the line $\{z=\rho e^{i\theta} : \rho>0, \theta=\alpha\}$).
Consider the map $\Phi:\overline{A}\to\Phi(\overline{A})$
given by $\Phi(z):=z^\frac{\pi}{\alpha}=\rho^{\frac{\pi}{\alpha}}e^{i\frac{\pi}{\alpha}\theta}$, where $z=\rho e^{i\theta}$.
The map $\Phi$ is of class $C^{1,\frac{\pi}{\alpha}-1}(\overline{A})$, and since it is conformal out of the origin,
the function $v:=u\circ\Phi^{-1}$ is harmonic in $\Phi(A)$
and satisfies a homogenous Neumann condition on $\Phi(\Gamma_1\cup\Gamma_2)$.
Moreover $\Phi(\Gamma_1\cup\Gamma_2)$ is a curve of class $C^{1,\gamma}$, hence by classical regularity results
(see, \emph{e.g.}, \cite[Theorem~7.49]{AFP}) $\nabla v$ has a $C^{0,\gamma}$ extension up to $\Phi(\Gamma_1\cup\Gamma_2)$,
with $C^{0,\gamma}$-norm bounded by a constant depending only on the $C^{1,\gamma}$-norm of $\Phi(\Gamma_2)$.
The conclusion follows since $u=v\circ\Phi$, using the regularity of $\Phi$.
\end{proof}

\begin{lemma} \label{lemma:Dirichletreg}
Let $A$ be an open subset of the unit ball $B_1$
such that $\partial A\cap B_1 = \Gamma_1 \cup \Gamma_2$,
where $\Gamma_1$ and $\Gamma_2$ are two curves of class $C^{1,\beta}$
meeting at the origin with an internal angle equal to $\frac{\pi}{2}$.
Let $u\in H^1(A)$ be a weak solution to
$$
\left\{
  \begin{array}{ll}
    \Delta u=f & \hbox{in }A \\
    \partial_\nu u=0 & \hbox{on }\Gamma_1 \\
    u=u_0 & \hbox{on }\Gamma_2
  \end{array}
\right.
\qquad\text{ or to }\qquad
\left\{
  \begin{array}{ll}
    \Delta u=f & \hbox{in }A \\
    \partial_\nu u=0 & \hbox{on }\Gamma_1\cup\Gamma_2
  \end{array}
\right.
$$
where $f\in L^\infty(A)$, and $u_0\in C^2(\overline{A})$ is such that $\partial_\nu u_0=0$ on $\Gamma_1$.
Then $\nabla u$ has a $C^{0,\beta}$ extension up to $\Gamma_1\cup\Gamma_2$,
with $C^{0,\beta}$-norm bounded by a constant depending only on $\|f\|_\infty$,
on the $C^{1,\beta}$-norm of $\Gamma_1$ and $\Gamma_2$,
and on $\|u_0\|_{C^2(\overline{A})}$ in the first case.
\end{lemma}

\begin{proof}
Let $u$ solve the first problem, and let $\tilde{u}:=u-u_0$.
Then $\tilde{u}$ is a solution to
$$
\left\{
  \begin{array}{ll}
    \Delta \tilde{u}=\tilde{f} & \hbox{in }A, \\
    \partial_\nu \tilde{u}=0 & \hbox{on }\Gamma_1, \\
    \tilde{u}=0 & \hbox{on }\Gamma_2,
  \end{array}
\right.
$$
where $\tilde{f}:=f-\Delta u_0$.
We can find a radius $\rho>0$ and a $C^{1,\beta}$ conformal mapping $\Phi$
in $\overline{A}\cap\overline{B}_\rho$ such that $\Phi(\Gamma_1)$ is a straight line meeting $\Phi(\Gamma_2)$ orthogonally.
Then the function $v:=\tilde{u}\circ\Phi^{-1}$ solves
$$
\left\{
  \begin{array}{ll}
    \Delta v=g & \hbox{in }\Phi(A), \\
    \partial_\nu v=0 & \hbox{on }\Phi(\Gamma_1), \\
    v=0 & \hbox{on }\Phi(\Gamma_2),
  \end{array}
\right.
$$
where $g:=(\tilde{f}\circ\Phi^{-1})|\det\nabla\Phi^{-1}|$.
By even reflection across $\Phi(\Gamma_1)$ and by applying classical regularity results,
we can conclude that  $\nabla v$ has a $C^{0,\beta}$ extension up to $\Phi(\Gamma_1\cup\Gamma_2)$,
with $C^{0,\beta}$-norm bounded by a constant depending only on $\|g\|_\infty$ and on the $C^{1,\beta}$-norm of $\Phi(\Gamma_1\cup\Gamma_2)$. Now the conclusion follows by using the regularity of the map $\Phi$.

The regularity for the solution to the second problem can be obtained by a similar (and, in fact, simpler) argument.
\end{proof}


\subsection{Proof of the density lower bound} \label{appendix:lowerbound}

This concluding subsection is entirely devoted to the proof of Theorem~\ref{teo:dlb}.
Most of the proofs are classical and very similar to those contained in \cite{BG}
(which in turn follow the approach of \cite{DCL}),
and for this reason we will just sketch them by describing only the main changes needed,
referring the interested reader to \cite{Bon} for details.

We start by observing that, if $w$ satisfies the hypotheses of the theorem,
the following energy upper bound holds in every ball $B_\rho(x)$ with $\rho\leq R_0$
(it can be easily deduced by comparing the energies of $w$ and of $w\chi_{\Omega'\setminus(B_\rho(x)\cap\Omega)}$):
\begin{equation}\label{eq:dlb0}
\F(w;B_\rho(x)\cap\Omega') \leq c_0\rho,
\end{equation}
where $c_0$ depends only on $R_0$, $\omega$, $u$ and $\Omega$.
In the following, $C$ will always denote a positive constant depending only on the previous quantities.
We now show that we can replace the Dirichlet condition in $\Omega'\setminus\Omega$ by a homogeneous boundary condition.

\begin{lemma} \label{lemma:dlb1}
Set $\tilde{w}:=w-u$.
Then $\tilde{w}\in\sbv(\Omega')$, $\tilde{w}=0$ in $\Omega'\setminus\Omega$,
and there exist $\eta>0$, $\tilde{\omega}>0$ (depending only on $\Omega$, $\omega$ and $u$)
such that for every $x\in\ombar\cap\mathcal{N}_{\eta}(\partial_D\Omega)$ and for every $\rho<\eta$
$$
\F(\tilde{w};B_\rho(x)\cap\Omega') \leq \F(v;B_\rho(x)\cap\Omega') + \tilde{\omega}\rho^\frac32
$$
whenever $v\in\sbv(\Omega')$ is such that $v=0$ in $\Omega'\setminus\Omega$ and $\{v\neq \tilde{w}\}\subset\subset B_\rho(x)$.
\end{lemma}

\begin{proof}
By choosing $\eta$ sufficiently small, we can guarantee that $\overline{S}_u\cap B_\rho(x)=\emptyset$
for each ball $B_\rho(x)$ as in the statement, hence $S_{\tilde{w}}\cap B_\rho(x)= S_w\cap B_\rho(x)$.
By comparing the energies of $w$ and $v+u$ we obtain
\begin{align*}
\F(\tilde{w};B_\rho(x)\cap\Omega')
\leq
\F(v;B_\rho(x)\cap\Omega')
+ 2\int_{B_\rho(x)\cap\Omega'} \nabla u \cdot (\nabla v - \nabla w)
+ 2\int_{B_\rho(x)\cap\Omega'}|\nabla u|^2 + \omega\rho^2.
\end{align*}
Now, using the fact that $\nabla u\in L^\infty$ and the upper bound \eqref{eq:dlb0}, we have
$$
2\int_{B_\rho(x)\cap\Omega'} |\nabla u|^2 \leq C\rho^2,
\qquad
-2\int_{B_\rho(x)\cap\Omega'} \nabla w\cdot\nabla u \leq C\rho^{\frac32},
$$
while for every $\e>0$ we have
$$
2\int_{B_\rho(x)\cap\Omega'} \nabla v\cdot\nabla u \leq
\e^2\|\nabla v\|^2_{L^2} + \frac{1}{\e^2}\|\nabla u\|^2_{L^2}
\leq \e^2\F(v;B_\rho(x)\cap\Omega') +\frac{C}{\e^2}\rho^2.
$$
It follows that
$$
\F(\tilde{w};B_\rho(x)\cap\Omega') \leq (1+\e^2)\F(v;B_\rho(x)\cap\Omega')
+ C \Bigl(1+\frac{1}{\e^2}\Bigr)\rho^2 + C\rho^{\frac32}.
$$
Defining the deviation from minimality of $\tilde{w}$ in a Borel set $B$ as
\begin{equation}\label{eq:dlb1}
\dev(\tilde{w};B) :=
\F(\tilde{w};B\cap\Omega') - \inf\bigl\{\F(v;B\cap\Omega'): v\in\sbv(\Omega'), \,
v=0\text{ in }\Omega'\setminus\Omega, \,
\{v\neq \tilde{w}\}\subset\subset B\bigr\},
\end{equation}
from the previous inequality we obtain,
by taking the infimum over all $v$,
\begin{align*}
\dev(\tilde{w};B_\rho(x))
&\leq \e^2 \F(\tilde{w};B_\rho(x)\cap\Omega') + C\Bigl(1+\frac{1}{\e^2}\Bigr)\rho^2 + C\rho^{\frac32} \\
&\leq c_0\e^2\rho + C\Bigl(1+\frac{1}{\e^2}\Bigr)\rho^2 + C\rho^{\frac32}
\leq \tilde{\omega}\rho^{\frac32},
\end{align*}
where we used \eqref{eq:dlb0} in the second inequality
and we choose $\e=\rho^{\frac14}$ in the last inequality.
\end{proof}

In the proof of the main decay property in Lemma~\ref{lemma:dlb5}
we will need to consider the blow-up in a sequence of balls whose centers converge to a point in $\overline{\partial_D\Omega}\cap\overline{\partial_N\Omega}$.
This situation is examined in the following lemma.

\begin{lemma} \label{lemma:dlb2}
Let $x_n\in\overline{\Omega}$, $x_n\to x_0\in\overline{\partial_D\Omega}\cap\overline{\partial_N\Omega}$, and $r_n\to0^+$.
Setting
\begin{equation} \label{eq:dlb2}
\Omega_n:=\frac{\Omega' - x_n}{r_n}\cap B_1, \qquad
D_n:=\frac{(\Omega'\setminus\Omega) - x_n}{r_n}\cap B_1,
\end{equation}
there exist $\delta_1,\delta_2\in[0,1]$ and a coordinate system such that (up to subsequences)
$$
\Omega_n\to\Omega_0 :=\{(\xi,\zeta)\in B_1: \xi<\delta_1\}, \quad
D_n\to D_0 :=\{(\xi,\zeta)\in B_1: \xi<\delta_1, \, \zeta>\delta_2\}
$$
in $L^1$. Moreover, the relative isoperimetric inequality holds in $\Omega_n$
with a constant which can be chosen independently of $n$
(and which will be denoted by $\gamma$).
Finally, assuming $\delta_2<1$, given $v\in W^{1,2}(\Omega_0)$ with $v=0$ in $D_0$
there exists a sequence $v_n\in W^{1,2}(B_1)$ such that $v_n\to v$ in $W^{1,2}(\Omega_0)$ and $v_n=0$ in $D_n$.
\end{lemma}

\begin{proof}
The first part of the lemma states an intuitive geometric fact
that can be proved rigourously arguing as in \cite[Lemma~6.4]{BG}.
%
The fact that the constant in the relative isoperimetric inequality
can be chosen uniformly for all the sets $\Omega_n$
follows from the fact that the boundaries of the sets $\Omega_n$ are close to the boundary of $\Omega_0$ in the $C^1$ sense.

Finally, we prove the last part of the statement, under the assumption $\delta_2<1$.
We extend $v$ to the set $\widetilde{\Omega}=\Omega_0\cup\{\zeta>\delta_2\}$
by setting $v=0$ outside $\Omega_0$,
and since $\widetilde{\Omega}$ satisfies the exterior cone condition we can find $\tilde{v}\in W^{1,2}(\R^2)$
such that $\tilde{v}_{|\widetilde{\Omega}}=v$.
Setting, for $(\xi,\zeta)\in B_1$,
$$
v_n(\xi,\zeta):= \tilde{v}(\xi,\zeta+a_n), \quad
a_n:=\sup_{(\xi,\zeta)\in\partial D_n\cap\Omega_n}|\zeta-\delta_2|\to0,
$$
we obtain a sequence with the desired properties.
\end{proof}

In the following compactness property, which is a consequence of the Poincar\'{e} inequality,
we adapt \cite[Proposition~7.5]{AFP} to our context.

\begin{lemma} \label{lemma:dlb3}
Let $x_n$ and $r_n$ be as in Lemma~\ref{lemma:dlb2}, and assume that $|D_n|\geq d_0>0$ for every $n$.
Let $u_n\in\sbv(\Omega_n)$, with $u_n=0$ a.e. in $D_n$, be such that
$$
\sup_n\int_{\Omega_n}|\nabla u_n|^2\,dx <\infty, \qquad
\lim_{n\to\infty}\hu(S_{u_n})=0.
$$
Setting $\bar{u}_n:=(u_n\wedge\tau_n^+)\vee\tau_n^-$, where
\begin{align*}
\tau_n^+ &:= \inf\{t\in[-\infty,+\infty]: |\{u_n<t\}|\geq |\Omega_n|- (2\gamma\hu(S_{u_n}))^2 \} , \\
\tau_n^- &:= \inf\{t\in[-\infty,+\infty]: |\{u_n<t\}|\geq (2\gamma\hu(S_{u_n}))^2 \} ,
\end{align*}
(here $\gamma$ is the constant introduced in Lemma~\ref{lemma:dlb2}),
one has that $\bar{u}_n=0$ in $D_n$ for $n$ large, and (up to subsequences)
$\bar{u}_n\to v\in W^{1,2}(\Omega_0)$ in $L^2_{\rm loc}(\Omega_0)$,
$u_n\to v$ a.e. in $\Omega_0$, and for every $\rho\leq1$
\begin{equation}\label{eq:dlb3}
\int_{\Omega_0\cap B_\rho}|\nabla v|^2\,dx \leq \liminf_{n\to\infty} \int_{\Omega_n\cap B_\rho}|\nabla\bar{u}_n|^2\,dx.
\end{equation}
\end{lemma}

\begin{proof}
The proof can be obtained by repeating word by word the proof of \cite[Lemma~6.1]{BG}.
We have only to be careful about the fact that in our context also the domain $\Omega_n$ depends on $n$
and is not fixed along the sequence.
The essential remark here is that the isoperimetric inequality holds in the sets $\Omega_n$
with a constant which can be chosen independent of $n$, as observed in Lemma~\ref{lemma:dlb2}.
\end{proof}

The following lemma is a variant of \cite[Theorem~7.7]{AFP}.
For $B\subset\R^2$ Borel set and $c>0$ we set
$$
\F(v,c;B):= \int_{B}|\nabla v|^2\,dx + c\hu(S_v\cap B).
$$

\begin{lemma} \label{lemma:dlb4}
Let $x_n$ and $r_n$ be as in Lemma~\ref{lemma:dlb2}, and assume that $|D_n|\geq d_0>0$ for every $n$.
Let $c_n>0$, $u_n\in\sbv(\Omega_n)$, with $u_n=0$ in $D_n$, be such that
$$
\sup_n\F(u_n,c_n;\Omega_n)<+\infty, \qquad \lim_{n\to+\infty} \dev_{D_n}(u_n,c_n;B_1)=0,
$$
$$
\lim_{n\to+\infty}\hu(S_{u_n})=0, \qquad u_n\to v\in W^{1,2}(\Omega_0) \text{ a.e. in }\Omega_0,
$$
where
\begin{align*}
\dev_{D_n}(v,c;B):=
\F&(v,c;B\cap\Omega_n) \\
&- \inf\bigl\{\F(w,c;B\cap\Omega_n): w\in\sbv(\Omega_n), \, w=0\text{ in }D_n, \, \{w\neq v\}\subset\subset B\bigr\}.
\end{align*}
Then
$$
\int_{\Omega_0}|\nabla v|^2\,dx \leq \int_{\Omega_0}|\nabla w|^2\,dx
$$
for every $w\in W^{1,2}(\Omega_0)$ such that $w=0$ in $D$ and $\{v\neq w\}\subset\subset B_1$,
and
$$
\lim_{n\to+\infty} \F(u_n,c_n;\Omega_n\cap B_\rho) = \int_{\Omega_0\cap B_\rho} |\nabla v|^2\,dx
\qquad\text{for every }\rho\in(0,1).
$$
\end{lemma}

\begin{proof}
Also in this case the proof is analogous to the one of \cite[Proposition~6.2]{BG},
and uses the auxiliary result contained in Lemma~\ref{lemma:dlb3}
in order to overcome the technical difficulties due to the fact of working in a variable domain.
We remark that in the proof it is essential to use the last part of Lemma~\ref{lemma:dlb2},
which allows us to approximate any competitor $w$ for $v$, vanishing in the limit domain $D$,
with functions $w_n$ vanishing in $D_n$, for which the quasi-minimality of $v_n$ can be exploited.
Taking into account these observations, it is straightforward to check that the proof of \cite[Proposition~6.2]{BG}
yields the conclusion also in our case.
\end{proof}

The following lemma contains the main decay property used to prove Theorem~\ref{teo:dlb}.

\begin{lemma} \label{lemma:dlb5}
There exists a positive constant $C$ such that
for every $\tau\in(0,1)$ there exist $\e(\tau)>0$, $\theta(\tau)>0$ and $r(\tau)>0$ with the property that
for every $x\in\ombar$ and $\rho\leq r(\tau)$,
whenever $v\in\sbv(\Omega'\cap B_{\rho}(x))$ is such that $v=0$ in $(\Omega'\setminus\Omega)\cap B_\rho(x)$,
$$
\hu(S_v\cap B_{\rho}(x)\cap\Omega')<\e(\tau)\rho, \quad
\dev(v;B_\rho(x)) < \theta(\tau)\F(v;B_\rho(x)\cap\Omega')
$$
(the deviation from minimality is defined as in \eqref{eq:dlb1}) then
$$
\F(v;B_{\tau\rho}(x)\cap\Omega') \leq C\tau^2\F(v;B_\rho(x)\cap\Omega').
$$
\end{lemma}

\begin{proof}
By choosing $C$ large enough, we can assume without loss of generality that $\tau<\frac14$.
The proof is by a contradiction argument: let $\e_n\to0$, $\theta_n\to0$, $r_n\to0$, $x_n\in\ombar$,
$v_n\in\sbv(B_{r_n}(x_n)\cap\Omega')$, $v_n=0$ in $(\Omega'\setminus\Omega)\cap B_{r_n}(x_n)$, be such that
$$
\hu(S_{v_n}\cap B_{r_n}(x_n)\cap\Omega')=\e_n r_n, \quad
\dev(v_n;B_{r_n}(x_n)) = \theta_n \F(v_n;B_{r_n}(x_n)\cap\Omega'),
$$
and
$$
\F(v_n;B_{\tau r_n}(x_n)\cap\Omega') > C\tau^2\F(v_n;B_{r_n}(x_n)\cap\Omega'),
$$
where $C$ will be chosen later. By a change of variables, we set
$$
w_n(y):=r_n^{-\frac12} \, {c_n}^{\frac12} \, v_n(x_n+r_ny),
\qquad
c_n:=\frac{r_n}{\F(v_n;B_{r_n}(x_n)\cap\Omega')}.
$$
We obtain a sequence $w_n\in\sbv(\Omega_n)$ such that
$\F(w_n,c_n;\Omega_n)=1$,
$\dev_{D_n}(w_n,c_n;B_1)=\theta_n$,
$\hu(S_{w_n}\cap\Omega_n)=\e_n$, and
$$
\F(w_n,c_n;B_{\tau}\cap\Omega_n) > C\tau^2
$$
(here $\Omega_n$ and $D_n$ are defined as in \eqref{eq:dlb2}).
Up to subsequences, $x_n\to x_0$, and we are in one of the following cases:
\begin{itemize}
  \item $x_0\in\Omega$: in this case the balls $B_{r_n}(x_n)$ are contained in $\Omega$ for $n$ large, hence the boundary does not play any role and the contradiction follows from \cite[Lemma~7.14]{AFP};
  \item $x_0\in\partial_D\Omega$: the balls $B_{r_n}(x_n)$ intersect only the Dirichlet part of the boundary for $n$ large, and the contradiction follows from \cite[Lemma~6.6]{BG};
  \item $x_0\in\partial_N\Omega$: we have that $\Omega_n\to\Omega_0=\{(\xi,\zeta)\in B_1:\xi<\delta_1\}$ for some $\delta_1\in[0,1]$ (in a suitable coordinate system) and $D_n=\emptyset$ for $n$ large enough. Adapting Lemma~\ref{lemma:dlb3} and Lemma~\ref{lemma:dlb4} to this situation (in which the Dirichlet condition does not play any role) we have that, up to further subsequences, $w_n-m_n\to w$ almost everywhere in $\Omega_0$, where $m_n$ are medians of $w_n$ in $\Omega_n$ and $w\in W^{1,2}(\Omega_0)$, with$$\int_{\Omega_0}|\nabla w|^2\leq\liminf_n\int_{\Omega_n}|\nabla w_n|^2\leq1.$$In addition, $w$ is harmonic in $\Omega_0$ and satisfies a homogeneous Neumann condition on $\{(\xi,\zeta):\xi=\delta_1\}$, and hence (by the decay properties of harmonic functions)
\begin{align*}
C\tau^2
\leq\lim_{n\to+\infty} \F(w_n,c_n;B_\tau\cap\Omega_n)
= \int_{B_\tau\cap\Omega_0}|\nabla w|^2
\leq 8\tau^2\int_{B_{\frac12}\cap\Omega_0} |\nabla w|^2 \leq 8\tau^2
\end{align*}
      which is a contradiction if we take $C>8$.
  \item $x_0\in\overline{\partial_D\Omega}\cap\overline{\partial_N\Omega}$: in this case we are under the assumptions of Lemma~\ref{lemma:dlb2}. If $\delta_2\in(\frac12,1]$, then $B_{\frac12}\cap D_n=\emptyset$ for $n$ large enough, and we can argue exactly as in the previous case, in the ball $B_{\frac12}$. It remains only to deal with the case $\delta_2\in[0,\frac12]$.
\end{itemize}
To get a contradiction also in the case $\delta_2\in[0,\frac12]$, observe first that $|D_n|\geq d_0>0$.
We can apply Lemma~\ref{lemma:dlb3} and Lemma~\ref{lemma:dlb4} to deduce that, up to subsequences,
$w_n\to w_\infty\in W^{1,2}(\Omega_0)$ a.e. in $\Omega_0$, with $w_\infty=0$ in $D$,
$$
\int_{\Omega_0}|\nabla w_\infty|^2\leq\liminf_{n\to\infty}\int_{\Omega_n}|\nabla w_n|^2\leq1.
$$
Moreover for every $w\in W^{1,2}(\Omega_0)$ such that $w=0$ in $D$ and $\{w\neq w_\infty\}\subset\subset B_1$
$$
\int_{\Omega_0}|\nabla w_\infty|^2 \leq \int_{\Omega_0} |\nabla w|^2,
$$
and
$$
\F(w_n,c_n;B_r\cap\Omega_n)\to\int_{B_r\cap\Omega_0}|\nabla w_\infty|^2
\qquad\text{for every }r\in(0,1).
$$
If $\tilde{w}_\infty$ is the harmonic function in $B_1$
obtained by applying firstly an even reflection of $w_\infty$ across $\{(\xi,\zeta):\xi=\delta_1\}$,
and then an odd reflection across $\{(\xi,\zeta):\zeta=\delta_2\}$,
we conclude, by using the decay properties of harmonic functions, that
\begin{align*}
C \tau^2
& \leq \lim_{n\to\infty} \F(w_n,c_n;B_\tau\cap\Omega_n)
= \int_{B_\tau\cap\Omega_0} |\nabla w_\infty|^2
\leq \int_{B_\tau} |\nabla\tilde{w}_\infty|^2 \\
& \leq (2\tau)^2\int_{B_{\frac12}} |\nabla\tilde{w}_\infty|^2
\leq 4(2\tau)^2\int_{B_{\frac12}\cap\Omega_0} |\nabla w_\infty|^2
\leq 16\tau^2,
\end{align*}
and this is a contradiction if we chose $C>16$.
\end{proof}

We have now all the ingredients to conclude the proof of Theorem~\ref{teo:dlb}.

\begin{proof}[Proof of Theorem~\ref{teo:dlb}]
Let $\eta$ be given by Lemma~\ref{lemma:dlb1}.
We first observe that the density lower bound holds in any ball $B_\rho(x)$
with $x\in\ombar\setminus\mathcal{N}_\eta(\partial_D\Omega)$ and $\rho\leq\rho_0$
(for some $\rho_0<\eta$ depending only on $\omega$, $u$ and $\Omega$):
indeed, in this case the Dirichlet boundary condition does not play any role, and the result is classical.
It is then sufficient to prove the lower bound for the function $\tilde{w}$ defined in Lemma~\ref{lemma:dlb1}
in balls $B_\rho(x)$ centered at points $x\in\overline{S}_{\tilde{w}}\cap\mathcal{N}_\eta(\partial_D\Omega)$,
since in such balls $S_w\cap B_\rho(x)=S_{\tilde{w}}\cap B_\rho(x)$ if $\rho<\eta$.

In turn, in this case the conclusion follows by repeating exactly the proof of \cite[Theorem~3.4]{BG}
(with the particular values of the parameters $N=2$, $s=\frac12$ and $\beta=1$),
using the decay property proved in our context in Lemma~\ref{lemma:dlb5}.
\end{proof}

\end{section}


\bigskip
{
}

\end{document}